\documentclass[11pt]{article}

\usepackage{amsmath, amssymb, amsthm, color}

\theoremstyle{plain}
\newtheorem{thm}{Theorem}[section]
\newtheorem{prop}[thm]{Proposition}
\newtheorem{lem}[thm]{Lemma}
\newtheorem*{lem*}{Lemma}
\newtheorem{defn}[thm]{Definition}

\newtheorem*{rem}{Remark}

\newcommand{\BBE}{\mathbb{E}}
\newcommand{\BFP}{\mathbf{P}}

\date{}
\title{\vspace{-0.7cm}Compatible Hamilton cycles in random graphs}

\author{
Michael Krivelevich \thanks{School of Mathematical Sciences, Raymond and Beverly Sackler Faculty of Exact Sciences, Tel Aviv University, Tel Aviv,
6997801, Israel. Email: krivelev@post.tau.ac.il.
Research supported in part by USA-Israel BSF Grant 2010115 and by grant 912/12 from the Israel Science Foundation.}
\and
Choongbum Lee \thanks{Department of Mathematics,
MIT, Cambridge, MA 02139-4307. Email: cb\_lee@math.mit.edu.
Research supported in part by NSF Grant DMS-1362326.}
\and
Benny Sudakov \thanks{Department of Mathematics, ETH, 8092 Zurich, Switzerland.
Email: benjamin.sudakov@math.ethz.ch. Research supported in part by SNSF grant 200021-149111 and by a USA-Israel BSF grant.}
}

\begin{document}

\maketitle

\begin{abstract}
A graph is Hamiltonian if it contains a cycle passing through every vertex.
One of the cornerstone results in the theory of random graphs asserts that for edge probability $p \gg \frac{\log n}{n}$, the random graph $G(n,p)$
is asymptotically almost surely Hamiltonian. We obtain the following strengthening of this result.
Given a graph $G=(V,E)$, an {\em incompatibility system} $\mathcal{F}$ over $G$ is a family $\mathcal{F}=\{F_v\}_{v\in V}$ where for every $v\in V$, the set $F_v$ is a set of unordered pairs $F_v \subseteq \{\{e,e'\}: e\ne e'\in E, e\cap e'=\{v\}\}$.
An incompatibility system is {\em $\Delta$-bounded} if for every vertex $v$
and an edge $e$ incident to $v$, there are at most $\Delta$ pairs in $F_v$
containing $e$.
We say that a cycle $C$ in $G$ is {\em compatible} with $\mathcal{F}$
if every pair of incident edges $e,e'$ of $C$ satisfies $\{e,e'\} \notin F_v$. This notion is partly motivated by
a concept of transition systems defined by Kotzig in 1968, and can be used as a quantitative measure of
robustness of graph properties. We prove that there is a constant $\mu>0$ such that the random graph $G=G(n,p)$ with $p(n) \gg \frac{\log n}{n}$ is asymptotically almost surely such that
for any $\mu np$-bounded incompatibility system $\mathcal{F}$ over $G$, there is
a Hamilton cycle in $G$ compatible with $\mathcal{F}$. We also prove that for larger edge probabilities $p(n)\gg \frac{\log^8n}{n}$, the parameter $\mu$ can be taken to be any constant smaller than $1-\frac{1}{\sqrt 2}$.
These results imply in particular that typically in $G(n,p)$ for $p \gg \frac{\log n}{n}$,
for any edge-coloring in which each color appears at most $\mu np$ times
at each vertex, there exists a properly colored Hamilton cycle.
Furthermore, our proof can be easily modified to show that for any edge-coloring of such a random graph in which
each color appears on at most $\mu np$ edges, there exists a Hamilton cycle in which all edges have distinct colors (i.e., a rainbow Hamilton cycle).
\end{abstract}

\section{Introduction}

A \emph{Hamilton cycle} in a graph $G$ is a cycle passing through each vertex of
$G$, and a graph is \emph{Hamiltonian} if it contains
a Hamilton cycle. Hamiltonicity, named after Sir Rowan Hamilton who studied it in the 1850s, is an important
and extensively studied concept in graph theory. It is well known that deciding Hamiltonicity is an NP-complete
problem and thus one does not expect a simple sufficient and necessary
condition for Hamiltonicity. Hence the study of Hamiltonicity has been concerned with looking for simple sufficient conditions
implying Hamiltonicity. One of the most important results in this direction
is Dirac's theorem asserting that all $n$-vertex graphs, $n\ge 3$, of minimum
degree at least $\frac{n}{2}$ contains a Hamilton cycle. Sufficient
conditions for Hamiltonicity often provide good indication towards
similar results for more general graphs.
For example, P\'osa and Seymour's conjecture on the existence of
powers of Hamilton cycles, and Bollob\'as and Koml\'os's conjecture
on the existence of general spanning subgraphs of small bandwidth, are both (highly non-trivial) generalizations of Dirac's theorem (both conjectures have been settled,
see \cite{KoSaSz98, BoScTa}).

A \emph{binomial random graph} $G(n,p)$ is a probability space of
graphs on $n$ vertices where each pair of vertices form an edge
independently with probability $p$. With some abuse of terminology, we will use $G(n,p)$ to denote both the probability space and a random graph drawn from it. We say that $G(n,p)$
possesses a graph property $\mathcal{P}$ \emph{asymptotically almost surely}
(or \emph{a.a.s.} in short) if the probability that $G(n,p)$ has
$\mathcal{P}$ tends to $1$ as $n$ tends to infinity. Early results
on Hamiltonicity of random graphs were proved by P\'osa \cite{Posa},
and Korshunov \cite{Korshunov}. Improving on these
results, Bollob\'as \cite{Bollobas84}, and Koml\'os
and Szemer\'edi \cite{KoSz} proved that if $p\ge(\log n+\log\log n+\omega(n))/n$
for any function $\omega(n)$ that goes to infinity together with
$n$, then $G(n,p)$ is a.a.s. Hamiltonian. The range of $p$ cannot
be improved, since if $p\le(\log n+\log\log n-\omega(n))/n$, then
$G(n,p)$ a.a.s. has a vertex of degree at most one.
Hamiltonicity of random graphs has been studied in great depth,
and there are many beautiful results on the topic.

\medskip{}

Recently there has been increasing interest in the study of
robustness of graph properties, aiming to strengthen classical results
in extremal and probabilistic combinatorics.
For example, consider the property of being Hamiltonian. By Dirac's theorem, we know that all $n$-vertex graphs
of minimum degree at least $\frac{n}{2}$ (which we refer to as \emph{Dirac
graph}s) are Hamiltonian. To measure the robustness of this theorem, we
can ask questions such as: ``How many Hamilton cycles must a Dirac graph
contain?'', ``What is the critical bias of the Maker-Breaker Hamiltonicity game
played on a Dirac graph?'', or ``When does a random subgraph of a
Dirac graph typically contain a Hamilton cycle?'' (see \cite{CuKa, KrLeSu14}). Note
that an answer to each question above in some sense defines a measure of robustness
of a Dirac graph with respect to Hamiltonicity. Moreover, Dirac's theorem
itself can be considered as measuring robustness  of Hamiltonicity of complete graphs,
where we measure the maximum number of edges one can delete from each
vertex of the complete graph while maintaining Hamiltonicity (see
\cite{SuVu} for further discussion).

\medskip{}

In this paper, we are interested in yet another type of robustness measure, and study
the robustness of Hamiltonicity with respect to this measure.
\begin{defn}
Let $G=(V,E)$ be a graph.

\vspace{-0.3cm}
\begin{itemize}
  \setlength{\itemsep}{0pt} \setlength{\parskip}{0pt}
  \setlength{\parsep}{0pt}
\item[(i)] An {\em incompatibility system} $\mathcal{F}$ over $G$ is a family $\mathcal{F}=\{F_v\}_{v\in V}$ where for every $v\in V$, the set $F_v$ is a set of unordered pairs $F_v \subseteq \{\{e,e'\}: e\ne e'\in E, e\cap e'=\{v\}\}$.

\item[(ii)] If $\{e,e'\} \in F_v$ for some edges $e,e'$ and vertex $v$, then we
say that $e$ and $e'$ are \emph{incompatible} in $\mathcal{F}$.
Otherwise, they are \emph{compatible} in $\mathcal{F}$. A subgraph $H \subseteq G$
is \emph{compatible} in $\mathcal{F}$, if all its pairs of edges $e$ and $e'$ are compatible.
\item[(iii)]  For a positive integer $\Delta$, an incompatibility system $\mathcal{F}$
is {\em $\Delta$-bounded} if for each vertex $v \in V$ and an edge $e$ incident to $v$, there are at most $\Delta$ other edges $e'$ incident to $v$ that are
incompatible with $e$.
\end{itemize}
\end{defn}

The definition is motivated by two concepts in graph theory.
First, it generalizes \emph{transition systems} introduced by Kotzig \cite{Kotzig68}
in 1968, where a transition system is a $1$-bounded
incompatibility system.
Kotzig's work was motivated by a problem of Nash-Williams on cycle covering
of Eulerian graphs
(see, e.g. Section 8.7 of \cite{Bondy95}).

Incompatibility systems and compatible Hamiton cycles also generalize
the concept of properly colored Hamilton cycles in edge-colored graphs,
The problem of finding properly colored Hamilton cycles in edge-colored
graph was first introduced by Daykin \cite{Daykin}. He asked if there
exists a constant $\mu$ such that for large enough $n$,
there exists a properly colored Hamilton cycle in every
edge-coloring of a complete graph $K_{n}$ where each vertex has at
most $\mu n$ edges incident to it of the same color (we refer to
such coloring as a \emph{$\mu n$-bounded edge coloring}).
Daykin's question has been answered
independently by Bollob\'as and Erd\H{o}s \cite{BoEr76} with $\mu=1/69$,
and by Chen and Daykin \cite{ChDa} with $\mu=1/17$. Bollob\'as
and Erd\H{o}s further conjectured that all $(\lfloor \frac{n}{2}\rfloor-1)$-bounded
edge coloring of $K_{n}$ admits a properly colored Hamilton cycle.
After subsequent improvements by Shearer \cite{Shearer} and by Alon and
Gutin \cite{AlGu}, Lo \cite{Lo12} recently settled the conjecture asymptotically,
proving that for any positive $\varepsilon$, every $(\frac{1}{2}-\varepsilon)n$-bounded
edge coloring of $E(K_n)$ admits a properly colored Hamilton cycle.

Note that a $\mu n$-bounded edge coloring naturally defines $\mu n$-bounded
incompatibility systems, and thus the question mentioned above can
be considered as a special case of the problem of finding compatible
Hamilton cycles. However, in general, the restrictions introduced
by incompatibility systems need not come from edge-colorings of graphs,
and thus the results on properly colored Hamilton cycles
do not necessarily generalize easily to incompatibility systems.

\medskip{}

In this paper, we study compatible Hamilton cycles in random graphs.
We present two results.

\begin{thm} \label{thm:sparse_p}
There exists a positive real $\mu$ such
that for $p\gg\frac{\log n}{n}$, the graph $G=G(n,p)$ a.a.s. has
the following property. For every $\mu np$-bounded incompatibility
system defined over $G$, there exists a compatible Hamilton cycle.
\end{thm}

Our result can be seen as an answer to a generalized version of
Daykin's question. In fact, we generalize it in two directions.
First, we replace properly colored Hamilton cycles by compatible
Hamilton cycles, and second, we replace the complete graph by
random graphs $G(n,p)$ for $p \gg \frac{\log n}{n}$ (note that
for $p=1$, the graph $G(n,1)$ is $K_n$ with probability $1$).
Since $G(n,p)$ a.a.s.~has no Hamilton cycles for $p \ll \frac{\log n}{n}$,
we can conclude that $\frac{\log n}{n}$ is a ``threshold function'' for having
such constant $\mu$.
The constant $\mu$ we obtain in Theorem \ref{thm:sparse_p}
is very small, and our second result improves this
constant for denser random graphs.

\begin{thm} \label{thm:dense_p}
For all positive reals $\varepsilon$, if $p\gg \frac{\log^8 n}{n}$,
then the graph $G=G(n,p)$ a.a.s.~has the following property. For
every $\Big(1-\frac{1}{\sqrt{2}}-\varepsilon\Big)np$-bounded incompatibility
system defined over $G$, there exists a compatible Hamilton cycle.
\end{thm}

In an edge-colored graph, we say that a subgraph is {\em rainbow} if all its edges have distinct colors. There is a vast literature on the branch of Ramsey theory where one seeks rainbow subgraphs in edge-colored 
graphs. Note that one can easily avoid rainbow copies by using a single color for all edges, and hence in order to find a rainbow subgraph one usually imposes some restrictions
on the distribution of colors. In this context, 
Erd\H{o}s, Simonovits and S\'os \cite{ErSiSo} and Rado \cite{Rado} developed anti-Ramsey theory where one attempts to determine the maximum number of colors that can be used to color the edges of the complete graph 
without creating a rainbow copy of a fixed graph. In a different direction, one can try to find a rainbow copy of a target graph by imposing global conditions on the coloring of the host graph.
For a real $\Delta$, we say that an edge-coloring of $G$ is {\em globally $\Delta$-bounded} if each color appears at most $\Delta$ times on the edges of $G$. 
In 1982, Erd\H{o}s, Ne\v set\v ril and R\"odl \cite{ENR} initiated 
the study of the problem of finding rainbow subgraphs in a globally $\Delta$-bounded coloring of graphs. One very natural question of this type is to find sufficient conditions for the existence of a rainbow 
Hamilton cycle in any globally $\Delta$-bounded coloring. 
Substantially improving on an earlier result of Hahn and Thomassen \cite{HaTh}, Albert, Frieze and Reed \cite{AlFrRe} proved the existence of a constant $\mu>0$ for which every globally $\mu n$-bounded coloring of 
$K_n$ (for large enough $n$) admits a rainbow Hamilton cycle. In fact, they proved a stronger statement 
asserting that for all graphs $\Gamma$ with vertex set $E(K_n)$ (the edge set of the complete graph) 
and maximum degree at most $\mu n$, there exists a Hamilton cycle in $K_n$ which is also an independent set in $\Gamma$.

It turns out that the proof technique used in proving Theorem \ref{thm:sparse_p} can be easily modified to give the following result, that extends the above to random graphs.

\begin{thm} \label{thm:subramsey}
There exists a constant $\mu > 0$ such that for $p \gg \frac{\log n}{n}$, 
the random graph $G=G(n,p)$ a.a.s.~has the following property.
Every globally $\mu np$-bounded coloring of $G$ contains a rainbow Hamilton cycle.
\end{thm}

Theorem \ref{thm:subramsey} is best possible up to the constant $\mu$ since one can
forbid all rainbow Hamilton cycles in a globally $(1+o(1))np$-bounded coloring by simply
coloring all edges incident to some fixed vertex with the same color.

\medskip

The proof of the three theorems will be given in the following sections.
In Section \ref{sec:first}, we prove Theorems \ref{thm:sparse_p} and \ref{thm:subramsey}.
Then in Section \ref{sec:second}, we prove Theorem \ref{thm:dense_p}.

\medskip{}

\noindent \textbf{Notation.} A graph $G=(V,E)$ is given
by a pair of its vertex set $V=V(G)$ and edge set $E=E(G)$.
For a set $X$, let $N(X)$ be the set of vertices incident
to some vertex in $X$. For a pair of disjoint vertex sets $X$ and $Y$, let
$E(X,Y) = \{ (x,y) \,|\, x \in X, y \in Y, \, \{x, y\} \in E\}$,
and define $e(X,Y) = |E(X,Y)|$. We define the {\em length} of a path 
as its number of edges.
When there are several graphs under consideration, to avoid
ambiguity, we use subscripts such as $N_{G}(X)$ to indicate the
graph that we are currently interested in.

Throughout the paper, we tacitly assume
that the number of vertices $n$ of the graph is large enough
whenever necessary. We
also omit floor and ceiling signs whenever they are not crucial.
All logarithms are natural.

\section{Proof of Theorems~\ref{thm:sparse_p} and \ref{thm:subramsey}} \label{sec:first}

To prove Theorem \ref{thm:sparse_p}, we find a compatible
Hamilton cycle by first finding a compatible subgraph
that is also a good expander graph.

\begin{defn}
For positive reals $k$ and $r$,
a graph $R$ is a $(k,r)$-expander if all sets $X\subseteq V(R)$ of
size at most $|X|\le k$ satisfy $|N(X)\setminus X| \ge r|X|$.
\end{defn}

Once we find an expander subgraph, we construct a Hamilton cycle by using
P\'osa's rotation-extension technique, which is a powerful tool
exploiting the expansion property of the graph.
The following definition captures the key concept that we will utilize.

\begin{defn}
Given a graph $R$ and a path $P$ defined over the same vertex set,
we say that an edge $\{v,w\}$ is a \emph{booster} for the pair $(P,R)$
if there exists a path of length $|P|-1$ in the graph $P\cup R$
whose two endpoints are $v$ and $w$.
\end{defn}

The following lemma is a well-known tool that is central to
many applications of the P\'osa's rotation-extension technique
(see, e.g., Lemma 8.5 of \cite{Bollobas01}).

\begin{lem} \label{lem:posa_rot_ext}
Suppose that $R \subseteq K_n$ is a $(k,2)$-expander and $P \subseteq K_n$
is a path that is of maximum length in the graph $P \cup R$. Then $K_n$ contains
at least $\frac{(k+1)^2}{2}$ boosters for the pair $(P,R)$.
\end{lem}

\subsection{Proof of Theorem \ref{thm:sparse_p}}

In this subsection, we state our main lemmas without proof and prove
Theorem \ref{thm:sparse_p} using these lemmas. The proofs of the
lemmas will be given in the next subsection.

\begin{lem} \label{lem:sparse_rotation}
There exist positive constants $\mu$ and $d$ such that 
if $p\gg\frac{\log n}{n}$, then $G=G(n,p)$
a.a.s. has the following property. For every $\mu np$-bounded
incompatibility system $\mathcal{F}$ over $G$, there exists a subgraph
$R\subseteq G$ with the following properties:
\vspace{-0.3cm}
\begin{itemize}
  \setlength{\itemsep}{1pt} \setlength{\parskip}{0pt}
  \setlength{\parsep}{0pt}
\item[(i)] $R$ is compatible with $\mathcal{F}$,
\item[(ii)] $R$ is an $(\frac{n}{4},2)$-expander, and
\item[(iii)] $|E(R)| \le d n$.
\end{itemize}
\end{lem}

The previous lemma will be used for `rotating' paths, while our next
lemma will be used for `extending' cycles.

\begin{lem} \label{lem:sparse_extension}
For a positive constant $d$,
if $p\gg\frac{\log n}{n}$, then a.a.s. in $G=G(n,p)$, each pair of
subgraphs $(P,R)$ satisfying the conditions below has at
least $\frac{1}{64}n^{2}p$ boosters relative to it:
\vspace{-0.1cm}
\begin{itemize}
  \setlength{\itemsep}{1pt} \setlength{\parskip}{0pt}
  \setlength{\parsep}{0pt}
\item[(i)] $R$ is an $(\frac{n}{4},2)$-expander with $|E(R)| \le d n$, and
\item[(ii)] $P$ is a longest path in $P\cup R$.
\end{itemize}
\end{lem}

Theorem \ref{thm:sparse_p} easily follows from the two lemmas above.

\begin{proof}[Proof of Theorem \ref{thm:sparse_p}]
Let $\mu$ and $d$ be constants coming from Lemma \ref{lem:sparse_rotation}.
We assume that $\mu \le \frac{1}{256(d+1)}$ by reducing
its value if necessary.
Suppose that an instance $G$ of $G(n,p)$ that satisfies the conclusions
of Lemmas \ref{lem:sparse_rotation} and \ref{lem:sparse_extension}
is given. Let $R \subseteq G$ be an $(\frac{n}{4},2)$-expander whose existence
is guaranteed by Lemma \ref{lem:sparse_rotation}.

Given an incompatibility system $\mathcal{F}$ over $G$,
let $P \subseteq G$ be a path of maximum length among all paths
satisfying the following two conditions:
(i) $P \cup R$ is compatible with $\mathcal{F}$,
and (ii) $P$ is a longest path in $P \cup R$. Note that we are maximizing
over a non-empty collection, since a longest path in $R$ meets
the criteria.

By Lemma \ref{lem:sparse_extension}, the graph $G$ contains at least
$\frac{1}{64}n^2p$ boosters for the pair $(P,R)$.
Among these boosters, we would like to find a
booster $e$ such that $P \cup R \cup \{e\}$
is compatible with $\mathcal{F}$.
Towards this end, for each edge $e'=\{u,v\} \in E(P \cup R)$ we forbid
to use the edges incompatible with $e'$ as boosters.
Since $\mathcal{F}$ is $\mu np$-bounded,
each edge of $P \cup R$ forbids at most $2\mu np$ other edges.
Furthermore, since the number of edges in $P \cup R$ is at most
$n + |E(R)| \le (d + 1)n$, the
total number of edges forbidden is at most
\[
(n+|E(R)|)\cdot 2 \mu np\le 2(d+1)\mu n^{2}p < \frac{1}{64}n^2p,
\]
which is less than the number of boosters.
Therefore, we can find a booster $e$ such that $P \cup R \cup \{e\}$
is still compatible with $\mathcal{F}$.

Since $e$ is a booster for $(P, R)$, we see that there exists a
cycle $C$ of length $|P|$ in $P \cup R \cup \{e\}$. This cycle is compatible
with $\mathcal{F}$, since it is a subgraph of a graph compatible
with $\mathcal{F}$. Thus if $C$ is a Hamilton cycle, then we are done.
Otherwise since all $(\frac{n}{4},2)$-expanders are connected, there exists
a vertex $v \notin V(C)$ and an edge $e' \in E(R)$ connecting $v$ to
$C$. By using this edge, we can extend the cycle $C$ to a path in $P \cup R \cup \{e\}$
that is longer than $P$. Thus if we define $P'$ as the longest path
in $P \cup R \cup \{e\}$, then since $P' \cup R \subseteq P \cup R \cup \{e\}$,
we see that $P' \cup R$ is compatible with $\mathcal{F}$,
and $P'$ is a longest path in $P' \cup R$. This contradicts the
fact that $P$ is chosen as a path of maximum length subject
to these conditions, and shows that $C$ is a Hamilton cycle.
\end{proof}

\subsection{Proof of lemmas}

We first state two well-known results in probabilistic combinatorics.
The first theorem is a form of Chernoff's inequality as appears
in \cite[Theorem 2.3]{Mcdiarmid}.

\begin{thm}
Let $X\sim Bi(n,p)$, where $Bi(n,p)$ denotes the binomial random variable with parameters $n$ and $p$. For any $s\le\frac{1}{2}np$
and $t\ge2np$, we have
\[
\mathbf{P}(X\le s)\le e^{-s/4}\quad\text{and}\quad\mathbf{P}(X\ge t)\le e^{-3t/16}.
\]
Moreover, for all $0 < \varepsilon < \frac{1}{2}$ we have,
\[
\mathbf{P}(|X-np| > \varepsilon np) \le e^{-\Omega(\varepsilon^2np)}.
\]
\end{thm}

The second theorem is the standard local lemma (see, e.g., \cite{AlSp}).
\begin{thm}
Let $A_1, A_2, \cdots, A_n$ be events in an arbitrary probability space.
A directed graph $D=(V,E)$ on the set of vertices $V=[n]$
is called a dependency digraph for the events $A_1, \ldots, A_n$
if for each $i \in [n]$, the event $A_i$ is mutually independent of
all the events $\{A_j\,:\,(i,j)\notin E\}$. Suppose that $D=(V,E)$
is a dependency digraph for the above events and suppose that there are
real numbers $x_1,\ldots,x_n$ such that $0 \le x_i < 1$ and
$\mathbf{P}(A_i) \le x_i \prod_{(i,j)\in E} (1-x_j)$ for all $i \in [n]$.
Then $\mathbf{P}(\bigcap_{i=1}^{n} \overline{A_i}) \ge \prod_{i=1}^{n}(1-x_i)$.
In particular, with positive probability no event $A_i$ holds.
\end{thm}

The following lemma establishes several properties of $G(n,p)$ that
we need.

\begin{lem} \label{lem:property_g_n_p}
If $p\gg\frac{\log n}{n}$, then $G(n,p)$ a.a.s. satisfies the following
properties.
\vspace{-0.1cm}
\begin{itemize}
  \setlength{\itemsep}{0pt} \setlength{\parskip}{0pt}
  \setlength{\parsep}{0pt}
\item[(i)] all degrees are $(1+o(1))np$,
\item[(ii)] for all sets $X$ of size $|X|<(np^{4})^{-1/3}$, we have $e(X)\le8|X|$,
\item[(iii)] for all sets $X$ of size $|X|=t$ for $(np^{4})^{-1/3}\le t \le n$,
we have $e(X)\le t^{2}p\cdot\Big(\frac{n}{t}\Big)^{1/2}$, and
\item[(iv)] for disjoint sets $X$ and $Y$ satisfying $|X||Y|p\gg n$, we
have $e(X,Y)\ge\frac{1}{2}|X||Y|p$.
\end{itemize}
\end{lem}
\begin{proof}
We omit the proofs of Properties (i) and (iv), since they
follow easily from direct applications of Chernoff's inequality together
with the union bound.

The probability
of a fixed set $X$ of size $t$ to violate Property (ii) is at most
\[
{{\binom{t}{2}} \choose 8t}p^{8t}\le\Big(\frac{etp}{8}\Big)^{8t}<\Big(\Big(\frac{e}{8}\Big)^{8}\Big(\frac{t}{n}\Big)^{2}\Big)^{t}.
\]
Hence by the union bound, the probability of Property (ii) being violated
is at most
\[
\sum_{t=1}^{(np^{4})^{-1/3}}{n \choose t}\cdot\Big(\Big(\frac{e}{8}\Big)^{8}\Big(\frac{t}{n}\Big)^{2}\Big)^{t}<\sum_{t=1}^{n/(\log n)^{4/3}}\Big(e\cdot\Big(\frac{e}{8}\Big)^{8}\Big(\frac{t}{n}\Big)\Big)^{t}=o(1).
\]

Similarly, the probability of a fixed set $X$ of size $t$ to violate
Property (iii) is, by Chernoff's inequality, at most $e^{-c\cdot t^{2}p(n/t)^{1/2}}$
for some positive constant $c$. The function
\[
ctp\Big(\frac{n}{t}\Big)^{1/2}-2\log\Big(\frac{en}{t}\Big)
\]
is increasing for $t>0$ and for $t=(np^{4})^{-1/3}$ equals
$c(np)^{1/3}-O(\log np)>0$. Hence, for $t \ge (np^4)^{-1/3}$, we have that
\[
ct^{2}p\Big(\frac{n}{t}\Big)^{1/2}\ge2t\log\Big(\frac{en}{t}\Big).
\]
Thus by the union bound, the probability
of Property (iii) being violated is at most
\[
	\sum_{t=(np^{4})^{-1/3}}^{n}{n \choose t}e^{-ct^{2}p(n/t)^{1/2}}
	\le
	\sum_{t=(np^{4})^{-1/3}}^{n}e^{t\log(\frac{en}{t})-ct^{2}p(n/t)^{1/2}}=o(1).
\]
\end{proof}

We first prove Lemma \ref{lem:sparse_rotation} which we restate here
for the reader's convenience. The proof is based on a straightforward
application of local lemma, but is rather lengthy.

\begin{lem*}
There exist positive constants $\mu$ and $d$ such that 
if $p\gg\frac{\log n}{n}$, then $G=G(n,p)$
a.a.s.~has the following property. For every $\mu np$-bounded
incompatibility system $\mathcal{F}$ over $G$, there exists a subgraph
$R\subseteq G$ with the following properties:
\vspace{-0.1cm}
\begin{itemize}
  \setlength{\itemsep}{1pt} \setlength{\parskip}{0pt}
  \setlength{\parsep}{0pt}
\item[(i)] $R$ is compatible with $\mathcal{F}$,
\item[(ii)] $R$ is an $(\frac{n}{4},2)$-expander, and
\item[(iii)] $|E(R)| \le d n$.
\end{itemize}
\end{lem*}

\begin{proof}
Throughout the proof, let $c_0 = e, C_1 = C_2 = \frac{1}{2}, \alpha = \frac{1}{2}\Big(\frac{1}{20e}\Big)^{2}$, $d = 10(20e)^2$, and $\mu = \frac{1}{25c_0 d^2} < 10^{-11}$.
Condition on $G=G(n,p)$
satisfying the events of Lemma \ref{lem:property_g_n_p}. Suppose that we are given
a $\mu np$-bounded incompatibility system $\mathcal{F}$ over $G$. For all $v\in V(G)$,
independently (with repetition) choose $d$ random edges in $G$ incident to $v$,
and let $F(v)$ be the set of chosen edges. Let $R$
be the graph whose edge set is $\bigcup_{v} F(v)$. We claim that
$R$ has the properties listed above with positive probability. Note that 
Property (iii) trivially holds.

Let $t_{0}=\frac{1}{3}(np^{4})^{-1/3},\, t_{1}=\alpha n$, and $t_{2}=n/4$
for some constant $\alpha$  to be chosen later. There
are three types of events that we consider. First are events considering
compatibility of edges. For a pair of edges $e_{1}$ and $e_{2}$,
if $e_{1}\neq e_{2}$, then let $A(e_{1},e_{2})$ be the event that
both edges $e_{1}$ and $e_{2}$ are in $R$, and if $e_{1}=e_{2}=e$,
then let $A(e,e)$ be the event that the edge $e$ is chosen in
two different trials. Define
\[
\mathcal{A}=\{A(e_{1},e_{2})\,:\, e_{1},e_{2}\text{ are incompatible, or }e_{1}=e_{2}\}.
\]
Second are events considering expansion of small sets. For a set $W$,
let $B(W)$ be the event that $e_{R}(W)\ge\frac{d}{3}|W|$, and define, for $t_0\le t\le t_1$,
\[
\mathcal{B}_{t}=\{B(W)\,:\,|W|=3t\}\,.
\]
Third are events considering expansion of large sets. For a pair of
disjoint subsets $X$ and $Y$, let $C(X,Y)$ be the event that $e_{R}(X,Y)=0$,
and define, for $t_1\le t\le t_2$,
\[
\mathcal{C}_{t}=\{C(X,Y)\,:\, X\cap Y=\emptyset,\,|X|=t,\,|Y|=n-3t\}\,.
\]

We first prove that Properties (i) and (ii) hold if none of
the events in $\mathcal{A},\mathcal{B}_{t}$, and $\mathcal{C}_{t}$
happen. Property (i) obviously holds  if none of the events in $\mathcal{A}$
happens. Note that not having the events
$A(e,e)$ for all edges $e$ implies the fact that we obtain distinct
edges at each trial. Hence each set $X$ has
at least $d|X|$ distinct edges incident to it, and in particular, we have $|E(R)|=d n$.
For Property (ii), consider a set $X$ of size $|X|=t$
and assume that $|N_{R}(X) \setminus X|<2|X|$. Let $W$ be a superset of $X\cup N_{R}(X)$
of size exactly $3|X|$. By the fact mentioned above, we see that
$e_{R}(W)\ge d |X|>\frac{d}{3}|W|$.\medskip{}

1. If $t<t_{0}$, then since $d>24$, this contradicts the fact that
$G(n,p)$ has $e(W)\le8|W|$,

2. if $t_{0}\le t\le t_{1}$, then it contradicts the event $B(W)$,
and

3. if $t_{1}\le t\le n/4$, then we have $e(X,V\setminus W)=0$
and it contradicts the event $C(X,V\setminus W)$.\medskip{}

Hence in all three cases we arrive at a contradiction. Therefore if
none of the events in $\mathcal{A},\mathcal{B}_{t}$, and $\mathcal{C}_{t}$
holds, then we obtain all the claimed properties (i), (ii), and (iii).

\medskip{}

We will use the local lemma to prove that with positive probability none of the events
in $\mathcal{A},\mathcal{B}_{t}$, and $\mathcal{C}_{t}$ holds. Towards
this end, we define the dependency graph $\Gamma$ with vertex set $\mathcal{V} = \mathcal{A} \cup \bigcup_{t= t_0}^{t_1} \mathcal{B}_t \cup \bigcup_{t=t_1}^{n/4} \mathcal{C}_t$.
Note that the graph $R$ is determined by the outcome of $n$ events $\{F(v)\}_{v \in V(G)}$
(recall that $F(v)$ is a set of $d$ random edges incident to $v$).
We let $V_1, V_2 \in \mathcal{V}$ be adjacent in $\Gamma$ if there exists
a vertex $v \in V(G)$ such that both $V_1$ and $V_2$ are dependent on the outcome of $F(v)$,
i.e., if there exist edges $e_1$ and $e_2$ both incident to $v$ such that $V_1$ depends on $e_1$
and $V_2$ depends on $e_2$.
This graph can be used as the dependency digraph in the local lemma since
$F(v)$ and $F(w)$ are independent for all distinct vertices $v,w \in V(G)$.

In order to apply the local lemma, for each event in $\mathcal{V}$, 
we stimate its probability and the degree of the corresponding vertex in $\Gamma$.
In all cases, for
the dependency with events in $\mathcal{B}_{t}$ and $\mathcal{C}_{t}$,
we use the crude bounds $|\mathcal{B}_{t}|$ and $|\mathcal{C}_{t}|$.

\medskip{}

\noindent \textbf{Family $\mathcal{A}$} : For a fixed pair of intersecting edges $e_{1}$
and $e_{2}$, to bound the probability of the event $A(e_{1},e_{2})$,
first, for each $e_1$ and $e_2$, select from which vertex and on which trial that
edge was chosen (at most $2d$ choices for each edge), and second, compute
the probability that $e_{1}$ and $e_{2}$ were chosen at those trials
(probability at most $\Big(\frac{1+o(1)}{np}\Big)^{2}$). This gives
\[
\mathbf{P}(A(e_{1},e_{2}))\le (1+o(1)) \left(\frac{2d}{np}\right)^{2}.
\]
Define $x=c_{0}\left(\frac{2d}{np}\right)^{2}$ for some constant
$c_{0}\ge1$ to be chosen later (this parameter will be
used in the local lemma). To compute the number of neighbors of $A(e_1, e_2)$ 
in $\Gamma$ in the set $\mathcal{A}$, note that $A(e_{1},e_{2})$
is adjacent to $A(f_{1},f_{2})$ if and only if some two edges
$e_{i}$ and $f_{j}$ intersect. There are at most
three vertices in $e_{1}\cup e_{2}$,
each vertex has degree $(1+o(1))np$ in $G(n,p)$, and each edge has
at most $2\mu np$ other edges incompatible with it. Therefore
the number of neighbors of $A(e_{1},e_{2})$ in  $\mathcal{A}$ is at most
\[
3\cdot(1+o(1))np\cdot2\mu np=6\mu(np)^{2}(1+o(1)).
\]

\medskip

\noindent \textbf{Family $\mathcal{B}_{t}$} : Assume that $t_{0}\le t\le t_{1}$
and consider a fixed set $W$ of size $|W|=3t$. To bound the probability
of the event $B(W)$, first, choose $\frac{d}{3}|W|=d t$ edges among
the edges of $G(n,p)$ in $W$ (${e(W) \choose d t}$ choices), second, for each
chosen edge, select from which vertex and on which trial that edge
was chosen (at most $(2d)^{d t}$ choices altogether), third, compute the probability
that each choice became the particular edge of interest (probability
at most $\Big(\frac{1+o(1)}{np}\Big)^{d t}$). This gives
\[
\mathbf{P}(B(W))\le{e_G(W) \choose td}\cdot(2d)^{d t}\cdot\left(\frac{1+o(1)}{np}\right)^{d t}\le\left(\frac{e\cdot e_G(W)}{td}\cdot\frac{(2+o(1))d}{np}\right)^{d t},
\]
which by the assumption that $e_G(W)\le9t^{2}p\cdot\Big(\frac{n}{3t}\Big)^{1/2}$
(coming from Lemma~\ref{lem:property_g_n_p} (iii))
gives
\[
\mathbf{P}(B(W))\le\left(20e\cdot\Big(\frac{t}{n}\Big)^{1/2}\right)^{d t}.
\]
Define $y_{t}=e^{C_{1}t}\left(20e\cdot\Big(\frac{t}{n}\Big)^{1/2}\right)^{d t}$
for some positive constant $C_{1}$, and for later usage, note that
\begin{equation}
\sum_{t=t_{0}}^{t_{1}}|\mathcal{B}_{t}|\cdot y_{t}=\sum_{t=t_{0}}^{t_{1}}{n \choose 3t}\cdot
e^{C_{1}t}\left(20e\cdot\Big(\frac{t}{n}\Big)^{1/2}\right)^{d t}\le\sum_{t=t_{0}}^{t_{1}}\left(e^{C_{1}}\Big(\frac{en}{3t}\Big)^{3}\cdot(20e)^{d}\cdot\Big(\frac{t}{n}\Big)^{d/2}\right)^{t}=o(1),\label{eq:ll_b_union_bound}
\end{equation}
since $\alpha, d$, and $C_1$ satisfy $e^{C_{1}}(\frac{e}{3\alpha})^{3}\cdot(20e)^{d}\cdot\alpha^{d/2}<1$.

To compute the number of neighbors of $B(W)$ in 
$\mathcal{A}$ in $\Gamma$, note that the event $B(W)$ is adjacent to
$A(e_{1},e_{2})$ if $e_{1}$ or $e_{2}$ intersect $W$. Therefore,
the number of neighbors as above is at most
\[
|W|\cdot(1+o(1))np\cdot\mu np=(1+o(1))3\mu tn^{2}p^{2}.
\]

\medskip{}

\noindent \textbf{Family $\mathcal{C}_{t}$} : Assume that $t_{1}\le t\le n/4$
and consider a fixed pair of sets $X$ and $Y$ of sizes $|X|=t$
and $|Y|=n-3t$. For each vertex $v\in X$, let $d_G(v,Y)$ be the number
of neighbors of $v$ in $Y$ in $G(n,p)$. Then the probability that
$e_{R}(X,Y)=0$ is
\[
\mathbf{P}(C(X,Y))=\prod_{v\in X}\left(1-\frac{d_G(v,Y)}{(1+o(1))np}\right)^{d}
\le e^{-(1+o(1))d e_G(X,Y)/(np)}
\le e^{-d t(n-3t)/(3n)},
\]
where we used the assumption that $e_G(X,Y)\ge\frac{t(n-3t)}{2}p$. Define
$z_{t}=e^{C_{2}n}e^{-d t(n-3t)/(3n)}$ for some positive constant $C_{2}$,
and for later usage, note that
\begin{equation}
\sum_{t=t_{1}}^{t_{2}}|\mathcal{C}_{t}|\cdot z_{t}\le\sum_{t=t_{1}}^{t_{2}}2^{2n}\cdot e^{C_{2}n}\cdot e^{-d t(n-3t)/(3n)}=o(1),\label{eq:ll_c_union_bound}
\end{equation}
since $d\alpha(1-3\alpha)>3(C_{2}+2)$ and $d/16 >3(C_{2}+2)$.
For the degree of $C(X,Y)$ in $\mathcal{A}$ in $\Gamma$, we use
the crude bound
\[
|\mathcal{A}|\le n\cdot(1+o(1))np\cdot\mu np=(1+o(1))\mu n^{3}p^{2}.
\]

\medskip{}

To apply the local lemma, we must verify 
the following three inequalities (for appropriate choices of $t$
as determined by the sets $W$, $X$, and $Y$):
\begin{eqnarray*}
\mathbf{P}(A(e_{1},e_{2})) & \le & x\cdot(1-x)^{(1+o(1))6\mu (np)^{2}}\left(\prod_{t=t_{0}}^{t_{1}}(1-y_{t})^{|\mathcal{B}_{t}|}\right)\cdot\left(\prod_{t=t_{1}}^{n/4}(1-z_{t})^{|\mathcal{C}_{t}|}\right),\\
\mathbf{P}(B(W)) & \le & y_{t}\cdot(1-x)^{(1+o(1))3\mu tn^{2}p^{2}}\cdot\left(\prod_{t=t_{0}}^{t_{1}}(1-y_{t})^{|\mathcal{B}_{t}|}\right)\cdot\left(\prod_{t=t_{1}}^{n/4}(1-z_{t})^{|\mathcal{C}_{t}|}\right),\\
\mathbf{P}(C(X,Y)) & \le & z_{t}\cdot(1-x)^{(1+o(1))\mu n^{3}p^{2}}\cdot\left(\prod_{t=t_{0}}^{t_{1}}(1-y_{t})^{|\mathcal{B}_{t}|}\right)\cdot\left(\prod_{t=t_{1}}^{n/4}(1-z_{t})^{|\mathcal{C}_{t}|}\right),
\end{eqnarray*}
By (\ref{eq:ll_b_union_bound}) and (\ref{eq:ll_c_union_bound}),
we know that
\[
\left(\prod_{t=t_{0}}^{t_{1}}(1-y_{t})^{|\mathcal{B}_{t}|}\right)\cdot\left(\prod_{t=t_{1}}^{n/4}(1-z_{t})^{|\mathcal{C}_{t}|}\right)=1 - o(1).
\]
Recall that $y_{t}\ge e^{C_{1}t}\mathbf{P}(B(W))$ and $z_{t}\ge e^{C_{2}n}\mathbf{P}(C(X,Y))$.
Thus to have the above three inequalities, it suffices to prove that
\begin{eqnarray*}
1 & \le & (1+o(1))c_{0}\cdot(1-x)^{(1+o(1))6\mu (np)^{2}},\\
\forall t_{0}\le t\le t_{1},\,1 & \le & (1+o(1))e^{C_{1}t}(1-x)^{(1+o(1))3\mu tn^{2}p^{2}},\\
\forall t_{1}\le t\le t_{2},\,1 & \le & (1+o(1))e^{C_{2}n}(1-x)^{(1+o(1))\mu n^{3}p^{2}}.
\end{eqnarray*}

Note that $1-x=e^{-(1+o(1))x}$ and $x=c_{0}\cdot\left(\frac{2d}{np}\right)^{2}$.
Since $C_{1} \ge (1+o(1))12\mu d^2c_{0}$ and
$C_{2} \ge (1+o(1))4\mu d^{2}c_{0}$, 
the second and the third inequalities hold. Also, the first inequality
holds since the parameters are chosen so that
\[
c_{0}\cdot(1-x)^{(1+o(1))6\mu (np)^{2}}=c_{0}\cdot e^{-(1+o(1))24c_{0}d^{2}\mu}>1.
\qedhere
\]
\end{proof}

We now prove Lemma \ref{lem:sparse_extension} (restated here).

\begin{lem*}
For a positive constant $d$,
if $p\gg\frac{\log n}{n}$, then a.a.s. in $G=G(n,p)$, each pair of
subgraphs $(P,R)$ satisfying the conditions below has at
least $\frac{1}{64}n^{2}p$ boosters relative to it:
\vspace{-0.1cm}
\begin{itemize}
  \setlength{\itemsep}{1pt} \setlength{\parskip}{0pt}
  \setlength{\parsep}{0pt}
\item[(i)] $R$ is an $(\frac{n}{4},2)$-expander with $|E(R)| \le d n$, and
\item[(ii)] $P$ is a longest path in $P\cup R$.
\end{itemize}
\end{lem*}
\begin{proof}
To prove the lemma, we first fix a pair $(P,R)$ satisfying the conditions
given above, and estimate the probability that $G(n,p)$ contains
enough boosters for the pair.

Since $R$ is an $(\frac{n}{4},2)$-expander, Lemma \ref{lem:posa_rot_ext}
implies that $K_n$ contains at least $\frac{n^2}{32}$ boosters
for the pair $(P,R)$.
It thus follows that the expected number of boosters for $(P,R)$ in $G(n,p)$ is at least
$\frac{1}{32}n^{2}p.$ By Chernoff's inequality, with probability
at least $1-e^{-\Omega(n^{2}p)}$, we have at least $\frac{1}{64}n^{2}p$
boosters for $(P,R)$ in $G(n,p)$.

We use this estimate on the probability together with the union bound
to prove the lemma.
The total number of paths of all possible length is at most $n\cdot n!\le e^{n\log n}$,
and the total number of graphs $R$ that we must consider is at most
\[
{n^{2} \choose d n}\le e^{d n\log n}.
\]
Since $p\gg\frac{\log n}{n}$, we obtain our conclusion
by taking the union bound.
\end{proof}

To prove Theorem \ref{thm:subramsey}, given a globally $\mu np$-bounded edge coloring of $G(n,p)$, call a pair of edges {\em compatible} if they are of different color, and a subgraph $H \subseteq G$ {\em compatible} 
if it is rainbow. Re-define the events $A(e_1, e_2)$ accordingly.  
One can easily check that the proof given in this section establishes Theorem \ref{thm:subramsey}
after slightly changing the method used in estimating the degree in the dependency graph.
We omit the straightforward details.

\section{Proof of Theorem~\ref{thm:dense_p}} \label{sec:second}

In this section we present the proof of our second result which is based on several ideas.
First we use a strategy from \cite{AlGu} to transform the problem of finding a compatible Hamilton
cycle into a problem of finding a directed Hamilton cycle in an appropriately defined auxiliary graph.
This strategy requires a `well-behaved' perfect matching in our graph, which will be taken
using a `nibble method'.
Finally to complete the proof we use recent
resilience-type results on Hamiltonicity of random directed graphs proved in \cite{HeStSu} and \cite{FeNeNoPeSk}.

Before we delve into the (rather technical) details of the proof, let us provide a brief outline of our argument. Let $G=G(n,p)$ with $p \gg \frac{\log^8 n}{n}$. Assume a $\mu np$-bounded incomparability system  $\mathcal{F}$ over $G$ is given, and our aim is to find a Hamilton cycle in $G$ compatible with  $\mathcal{F}$. Assume for simplicity $n$ is even. Let $V=A\cup B$ be a random equipartition of $V(G)$ with $|A|=|B|=m=\frac{n}{2}$, and let $M$ be a randomly chosen perfect matching between $A$ and $B$ in $G$. 
Let $M=\{e_1,\ldots,e_{n/2}\}$, with $e_i=(a_i,b_i)$, $a_i\in A,b_i\in B$. We construct a Hamilton cycle
by further adding $n/2$ edges to $M$, while obeying compatibility. Define an auxiliary directed graph $D_G(M)$ as follows: its vertices are the edges of $M$, and $(e_i,e_j)$ is a directed edge of $D_G(M)$ if $\{b_i,a_j\}\in E(G)$. Since the edges of $D_G(M)$ are in one-to-one correspondence with the edges of $G$ between $A$ and $B$ outside $M$, we may consider $D_G(M)$ as a random directed graph on $m$ vertices with edge probability $p$.
Observe that a directed Hamilton cycle in $D_G(M)$ translates into a Hamilton cycle in $G$ in an obvious way.
To obtain a Hamilton cycle compatible with $\mathcal{F}$ through this correspondence, we remove some edges from $D_G(M)$. Consider an edge $e_i$ of $M$ (in its capacity as a vertex of $D_G(M)$). Let us see which directed edges $(e_i,e_j)$ leaving  $e_i$ in $D_G(M)$ need to be deleted. Those are edges for which $\{b_i,a_j\}$ is incompatible with $e_i$ according to  $\mathcal{F}$, and the number of such edges should be at most (about) $\mu mp$. In addition, we need to delete $(e_i,e_j)$ for which $\{b_i,a_j\}$ is compatible with $e_i$  (about $(1-\mu)$ proportion of edges) but incompatible with $e_j$ -- and the proportion of such edges should be about $\mu$. We expect these heuristic estimates to hold due to our random choice of $M$. Assuming these estimates, altogether we need to delete from $D_G(M)$ about $\mu mp+(1-\mu)\mu mp$ edges leaving $e_i$, and a similar amount of edges entering $e_i$. As mentioned above, the graph $D_G(M)$ is basically a random directed graph on $m$ vertices with edge probability $p$. At this stage, we invoke a recent result of Ferber et al. 
(\cite{FeNeNoPeSk}; Theorem \ref{thm:di_resilience} below), stating that if $p\gg \frac{\log^8 n}{n}$ then a random directed graph $D=D(n,p)$ is a.a.s. such that every subgraph of $D$ of minimum in- and out-degrees at least $(\frac{1}{2}+\varepsilon)np$ contains a directed Hamilton cycle. In order to able to apply this theorem to $D_G(M)$ we need to estimate from above the deleted in- and out-degrees at every vertex $e_i$ of $D_G(M)$ -- as we indicated above, and then to require that $\mu mp +(1-\mu)\mu mp\le (\frac{1}{2}-\varepsilon)mp$. So essentially we need to solve: $\mu+(1-\mu)\mu=\frac{1}{2}$ -- which gives us $\mu=1-\frac{1}{\sqrt{2}}$. It should be mentioned that this approach borrows some ideas from the argument of Alon and Gutin \cite{AlGu}, who also arrived at the same magical constant  $ 1-\frac{1}{\sqrt{2}}$ (but for the simpler case of the complete graph to start with).

We start with the following two definitions.

\begin{defn} \label{def:pm}
A \emph{perfect matching} on $n$ vertices over a partition
$A \cup B$ (for even $n$) or $A \cup B \cup \{v_*\}$ (for odd $n$),
with $|A|=|B|=\lfloor n/2\rfloor$, is a collection $\{(a_{i},b_{i})\}_{i=1}^{\lfloor n/2\rfloor}$
of disjoint pairs of vertices with $a_{i}\in A$ and $b_{i}\in B$ with the
following property:
\vspace{-0.1cm}
\begin{itemize}
  \setlength{\itemsep}{1pt} \setlength{\parskip}{0pt}
  \setlength{\parsep}{0pt}
\item[(i)] if $n$ is even, then $\{a_{i},b_{i}\}$ is an edge for every $i$,
and
\item[(ii)] if $n$ is odd, then $\{a_{i},b_{i}\}$ is an edge for $i=1,2,\cdots,(n-3)/2$,
and $(a_{(n-1)/2},v_{*},b_{(n-1)/2})$ is a path of length 2.
\end{itemize}
\end{defn}
In most cases, for a given vertex $a_{i}$, the edge incident to $a_{i}$
in the perfect matching is $\{a_{i},b_{i}\}$. However, when $n$
is odd and $i=(n-1)/2$, the edge incident to $a_{(n-1)/2}$ is $\{a_{(n-1)/2},v_{*}\}$.
This distinction is made for technical reasons and we recommend the
reader to assume that $n$ is even for the first time reading.
\begin{defn}
\label{def:digraph}Let $G$ be a graph and let $M$ be a perfect matching
$\{e_{i}\}_{i=1}^{\lfloor n/2\rfloor}$ for $e_{i}=(a_{i},b_{i})$ (and
$e_{\lfloor n/2\rfloor} = (a_{\lfloor n/2\rfloor}, v_*, b_{\lfloor n/2\rfloor})$ if $n$ is odd).
\vspace{-0.2cm}
\begin{itemize}
  \setlength{\itemsep}{1pt} \setlength{\parskip}{0pt}
  \setlength{\parsep}{0pt}
\item[(i)] Define $D_{G}(M)$ as the directed graph over the
vertex set $\{e_{1},e_{2},\cdots,e_{\lfloor n/2\rfloor}\}$,
where there is a directed edge from $e_{i}$ to $e_{j}$ if and only if
$\{b_i, a_j\} \in E(G)$.
\item[(ii)] For a given incompatibility system $\mathcal{F}$ over $G$, define
$D_{G}(M;\mathcal{F})$ as the subgraph of $D_{G}(M)$ obtained by
the following process: remove the directed edge $(e_{i},e_{j})$ whenever
$\{b_{i},a_{j}\}$ is incompatible with the edge of $M$ incident to
$b_{i}$, or with the edge of $M$ incident to $a_{j}$. Moreover, if $n$
is odd and $\{a_{\lfloor n/2\rfloor},v_{*}\}$ and $\{v_{*},b_{\lfloor n/2\rfloor}\}$ are
not compatible, then remove all edges in $D_{G}(M)$ incident to the
vertex $e_{\lfloor n/2\rfloor}$.
\end{itemize}
\end{defn}
We may also write $D(M)$ or $D(M;\mathcal{F})$, when $G$ is clear
from the context. The following proposition explains the motivation
behind the definition given above.

\begin{prop} \label{prop:compatible_to_digraph}
Let $G$ be a graph with an incompatibility
system $\mathcal{F}$, and let $M$ be a perfect matching in $G$. If
$D_{G}(M;\mathcal{F})$ contains a directed Hamilton cycle, then $G$
contains a Hamilton cycle compatible with $\mathcal{F}$.
\end{prop}
\begin{proof}
Let $M$ be given as $e_{1},e_{2},\ldots,e_{k}$, and without loss
of generality let $(e_{1},e_{2},\ldots,e_{k},e_{1})$ be the Hamilton
cycle in $D_{G}(M;\mathcal{F})$.

If $G$ has an even number of vertices, then by the definition of $D_{G}(M;\mathcal{F})$,
we see that
\[
(a_{1},b_{1},a_{2},b_{2},\ldots,a_{k},b_{k},a_{1})
\]
 is a Hamilton cycle in $G$. Moreover, the edge $\{a_{i},b_{i}\}$
is compatible with both $\{b_{i},a_{i+1}\}$ and $\{a_i, b_{i-1}\}$
(addition and subtraction of indices are modulo $k$)
by the definition of $D_{G}(M;\mathcal{F})$. Therefore, we found
a Hamilton cycle in $G$ compatible with $\mathcal{F}$.

If $G$ has an odd number of vertices, then as before, we see
that $(a_{1},b_{1},a_{2},b_{2},\ldots,a_{k},v_{*},b_{k},a_{1})$
is a Hamilton cycle in $G$. If $\{a_{k},v_{*}\}$ and $\{v_{*},b_{k}\}$
were not compatible, then by the definition of $D_{G}(M;\mathcal{F})$,
the vertex $e_{k}$ must be isolated in $D_{G}(M;\mathcal{F})$, contradicting
the fact that $D_{G}(M;\mathcal{F})$ is Hamiltonian. Therefore, the
pair is compatible. All other pairs are compatible as seen above.
\end{proof}

We prove the Hamiltonicity of $D_{G}(M;\mathcal{F})$ by carefully choosing
a perfect matching $M$ so that it satisfies the following two properties.

\begin{defn}
Let $\varepsilon$ be a fixed positive real, let $G$ be a given graph
with incompatibility system $\mathcal{F}$, and let $M$ be a
perfect matching in $G$.
\vspace{-0.2cm}
\begin{itemize}
  \setlength{\itemsep}{1pt} \setlength{\parskip}{0pt}
  \setlength{\parsep}{0pt}
\item[(i)] The pair $(G,M)$ is \emph{$\varepsilon$-di-ham-resilient} if
every subgraph of $D_{G}(M)$ of minimum in- and out-degrees
at least $(\frac{1}{2}+\varepsilon)\frac{\delta(G)}{2}$
contains a directed Hamilton cycle.
\item[(ii)] The triple $(G,M,\mathcal{F})$ is \emph{$\varepsilon$-typical}
if $D_{G}(M;\mathcal{F})$ has minimum in- and out-degrees at least
$\left(\frac{1}{2}+\varepsilon\right)\frac{\delta(G)}{2}$.
\end{itemize}
\end{defn}

The following lemma is the key ingredient of our proof, asserting the a.a.s.~existence of a
perfect matching in $G(n,p)$ for which
$(G,M)$ is $\varepsilon$-di-ham-resilient, and $(G,M,\mathcal{F})$
is $\varepsilon$-typical.

\begin{lem} \label{lem:gnp_resilient_typical}
Let $\varepsilon$ be a fixed positive real, and $p = \frac{\omega \log^8 n}{n}$
for some function $\omega = \omega(n) \le \log n$ that tends to infinity.
Then $G=G(n,p)$ a.a.s.~has the following property:
for every $\Big(1 - \frac{1}{\sqrt{2}} - 2\varepsilon\Big)np$-bounded
incompatibility system $\mathcal{F}$
over $G$, there exists a perfect matching $M \subseteq G$ such that
\vspace{-0.2cm}
\begin{itemize}
  \setlength{\itemsep}{1pt} \setlength{\parskip}{0pt}
  \setlength{\parsep}{0pt}
\item[(i)] $(G,M)$ is $\varepsilon$-di-ham-resilient and
\item[(ii)] $(G,M,\mathcal{F})$ is $\varepsilon$-typical.
\end{itemize}
\end{lem}

Next lemma allows us to restrict our attention to small values of $p$ as
in Lemma~\ref{lem:gnp_resilient_typical}. Its proof will be given in the following subsection.

\begin{lem} \label{lem:prob_transfer}
Suppose that $p_1$ and $p_2$ satisfying $1 \ge p_1 \ge p_2 \gg \frac{\log n}{n}$ are given.
If there exist positive real numbers $\alpha$ and $\varepsilon$ such that
$G(n,p_2)$ a.a.s.~contains a compatible Hamilton cycle for every
$(\alpha+2\varepsilon)np_2$-bounded incompatibility system, then
$G(n,p_1)$ a.a.s.~contains a compatible Hamilton cycle for every
$(\alpha+\varepsilon)np_1$-bounded incompatibility systems.
\end{lem}

The proof of Theorem \ref{thm:dense_p} easily follows from Lemma \ref{lem:gnp_resilient_typical}.

\begin{proof}[Proof of Theorem \ref{thm:dense_p}]
It suffices to prove the statement for $p = \frac{\omega \log^8 n}{n}$
for some $\omega = \omega(n) \le \log n$ that tends to infinity since larger
edge probabilities can be handled by Lemma \ref{lem:prob_transfer} below.
Let $\varepsilon$ be a given positive real and suppose that $G=G(n,p)$
satisfies the properties guaranteed by Lemma \ref{lem:gnp_resilient_typical}:
for every $(1-\frac{1}{\sqrt{2}} - 2\varepsilon)n$-bounded incompatibility system
$\mathcal{F}$ over $G$, there exists a perfect matching
$M$ for which the pair $(G,M)$ is $\varepsilon$-di-ham-resilient
and $(G,M,\mathcal{F})$ is $\varepsilon$-typical. These two properties imply that
$D_{G}(M;\mathcal{F})$ contains a directed Hamilton cycle, which by
Proposition \ref{prop:compatible_to_digraph} implies that $G$ contains a Hamilton
cycle compatible with $\mathcal{F}$.
\end{proof}

\subsection{Preliminaries}

Before proceeding to the proof of Lemma \ref{lem:gnp_resilient_typical},
we state some results  needed for our proof.
Let $D(n,p)$ be a random directed graph on $n$ vertices, in which for every pair $i \not= j$ the edge $i \rightarrow j$ appears independently with probability $p$.
The first theorem is a resilience-type result for Hamiltonicity of $D(n,p)$
which extends a classical result of Ghouila-Houri \cite{Ghouila}. It was first proved by
Hefetz, Steger, and Sudakov \cite{HeStSu} (for edge probabilities $p \geq n^{-1/2+o(1)}$) and then strengthened by Ferber, Nenadov, Noever, Peter, and Skoric
\cite{FeNeNoPeSk} to much smaller values of $p(n)$.

\begin{thm} \label{thm:di_resilience}
For all fixed positive reals $\varepsilon$, if $p\gg \frac{\log^8 n}{n}$
then $D(n,p)$ a.a.s.~has the following property:
every spanning subgraph of $D(n,p)$ of minimum in- and out-degrees
at least $(\frac{1}{2}+\varepsilon)np$ contains a directed Hamilton
cycle.
\end{thm}

We will often be considering events defined over the product of two probability spaces,
and the following simple lemma will be handy.

\begin{lem} \label{lem:two_space_to_one}
Let $X_1$ and $X_2$ be two random variables, and
suppose that there exists a set $A$ such that
$\mathbf{P}( (X_1, X_2) \in A) = 1 - x$ for some positive real $x$.
Let
\[ A_1 = \Big\{ a \,\,\Big|\,\, \mathbf{P}( (X_1, X_2) \in A \,|\, X_1 = a ) \ge 1 - \sqrt{x} \Big\}. \]
Then $\mathbf{P}(X_1 \in A_1) \ge 1 - \sqrt{x}$.
\end{lem}
\begin{proof}
Since
\[ x = \mathbf{P}\Big((X_1, X_2) \notin A\Big) \ge \mathbf{P}( X_1 \notin A_1) \cdot \sqrt{x}, \]
we have $\mathbf{P}(X_1 \notin A_1) \le \sqrt{x}$,
or equivalently $\mathbf{P}(X_1 \in A_1) \ge 1 - \sqrt{x}$.
\end{proof}

We prove Lemma~\ref{lem:prob_transfer} which, as seen in the previous subsection,
allows us to restrict our attention to sparse random graphs.
For two graphs $G_1 \supseteq G_2$ and an incompatibility
system $\mathcal{F}$ defined over $G_1$, we define the {\em incompatibility system
induced by $\mathcal{F}$ on $G_2$} as the incompatibility system
where two edges $e, e' \in E(G_2)$ are incompatible if and only if
they are in $\mathcal{F}$.

\begin{proof}
Let $G_1 = G(n,p_1)$ and let $G_2$ be a random subgraph of $G_1$ obtained by
retaining every edge independently with probability $\frac{p_2}{p_1}$.
Note that the distribution of the subgraph $G_2$ is identical to that of $G(n,p_2)$.

Let $\mathcal{R}$ be the collection of graphs that contain a compatible
Hamilton cycle for every $(\alpha+ 2\varepsilon)np_2$-bounded incompatibility
system. By the assumption of the lemma, we know that
\[
\mathbf{P}(G_2 \in \mathcal{R}) = 1-o(1).
\]
Let $\mathcal{R}_1$ be the collection of graphs $\Gamma$ such that
$\mathbf{P}(G_2 \in \mathcal{R} \,|\, G_1 = \Gamma) \ge \frac{1}{2}$.
By Lemma \ref{lem:two_space_to_one}, we see that
\[ \mathbf{P}(G_1 \in \mathcal{R}_1) = 1-o(1). \]
On the other hand, for each fixed $(\alpha+\varepsilon)np_1$-bounded
incompatibility system $\mathcal{F}$ over $G_1$, by Chernoff's inequality
and the union bound, with probability greater than $\frac{1}{2}$,
the incompatibility system induced by $\mathcal{F}$ on $G_2$
is $(\alpha+2\varepsilon)np_2$-bounded.

Therefore, if $G_1 \in \mathcal{R}_1$, then for every $(\alpha+\varepsilon)np_1$-bounded
incompatibility system $\mathcal{F}$ over $G_1$, there exists a subgraph $G_2' \subseteq G_1$
such that $\mathcal{F}$ induces an $(\alpha+2\varepsilon)np_2$-bounded incompatibility
system over $G_2'$, and $G_2' \in \mathcal{R}$. These two properties imply that $G_2'$
contains a Hamilton cycle compatible with $\mathcal{F}$, which in turn implies that
$G_1$ also contains such Hamilton cycle.
\end{proof}

\subsection{Proof of Lemma \ref{lem:gnp_resilient_typical}}

In this subsection, we prove Lemma \ref{lem:gnp_resilient_typical}.
The perfect matching $M$ in the statement of the lemma will be chosen according to some
random process that we denote by $\Phi$, i.e., $M = \Phi(G)$.
In fact we prove the following strengthening of Lemma \ref{lem:gnp_resilient_typical}.

\begin{lem} \label{lem:gnp_r_t_strong}
If $p = \frac{\omega_n \log^8 n}{n}$ for some $\omega_n \le \log n$ that tends
to infinity, then $G=G(n,p)$ a.a.s. has the following property.
For every $\Big(1 - \frac{1}{\sqrt{2}} - 2\varepsilon\Big)np$-bounded
incompatibility system $\mathcal{F}$ over $G$,
a random perfect matching $M = \Phi(G)$ satisfies each of the
following properties with probability $1 - o(1)$,
\begin{itemize}
  \setlength{\itemsep}{1pt} \setlength{\parskip}{0pt}
  \setlength{\parsep}{0pt}
	\item[(i)]
$(G,M)\text{ is \ensuremath{\varepsilon}-di-ham-resilient}$, and
	\item[(ii)]
$(G,M,\mathcal{F})\text{ is \ensuremath{\varepsilon}-typical}$.
\end{itemize}
\end{lem}

Note that Lemma \ref{lem:gnp_resilient_typical} immediately follows from Lemma \ref{lem:gnp_r_t_strong},
since the latter implies the a.a.s.~existence
of a particular instance of $M$ for which both Properties (i) and (ii) hold.

\medskip

Throughout the section, we assume that $\varepsilon$ is a given fixed positive real
(we may assume that $\varepsilon$ is small enough by decreasing its value if necessary),
and let $\delta = e^{-22\varepsilon^{-1} \ln \varepsilon^{-1}}$. Given $G = G(n,p)$,
we construct a perfect matching $\Phi(G)$ by the following algorithm
(we first give a description for even $n$).

\begin{itemize}
  \setlength{\itemsep}{1pt} \setlength{\parskip}{0pt}
  \setlength{\parsep}{0pt}
\item[1.] Take a random bipartite subgraph $H$ of $G$ by choosing a uniform bisection
$A \cup B$ and then taking each edge crossing the bisection
independently with probability $\frac{\varepsilon}{4}$. Initialize $H_0:=H$, $A_0:=A$, $B_0:=B$.

\item[2.] Repeat the following steps $T = \frac{\ln (4/\varepsilon)}{-\ln (1 - \delta)} \approx \frac{\ln (4/\varepsilon)}{\delta}$ times (start from $i=0$).
\vspace{-0.1cm}
\begin{itemize}
  \setlength{\itemsep}{1pt} \setlength{\parskip}{0pt}
  \setlength{\parsep}{0pt}

\item[2-1.] Given a bipartite graph $H_i$ with bipartition $A_i \cup B_i$
satisfying $n_i := |A_i| = |B_i|$ and $m_i := e(H_i)$, choose each edge
of $H_i$ independently with probability $\frac{\delta n_i}{m_i}$ to form a set of
edges $M_i^{(0)}$.

\item[2-2.] Let $M_i \subseteq M_i^{(0)}$ be the set of edges incident to
no other edges in $M_i^{(0)}$.

\item[2-3.] Remove the vertices incident to the edges in $M_i$ from $H_i$ to obtain $H_{i+1}$.
\end{itemize}
\vspace{-0.1cm}
\item[3.] Take an arbitrary perfect matching $M_T$ in the remaining graph $H_T$ and define
$\Phi(G)$ as the union of the matchings $M_0, \ldots, M_T$.
\end{itemize}

\begin{rem} If $n$ is odd, then for each $i=0,1,\ldots, T$,
the bipartite graphs $H_i$ will
have bipartition $A_i \cup B_i$ with $|A_i| = n_i + 1$ and $|B_i| = n_i$ for all
$i=0,1,\ldots, T$. In this case, in Step 3, first choose an edge
$\{v_*, a_{(n-1)/2}\}$ within $A_T$, and then choose $b_{(n-1)/2} \in B_T$
so that $\{v_*, a_{(n-1)/2}\}$ and $\{v_*, b_{(n-1)/2}\}$ are compatible.
Afterwards, find a perfect matching between $A_T \setminus \{v_*, a_{(n-1)/2}\}$
and $B_T \setminus \{b_{(n-1)/2}\}$.
\end{rem}

Since each $M_i$ forms a matching, the algorithm above
produces a sequence of balanced bipartite
graphs $H_i$ with vertex partition $A_i \cup B_i$ for $i=0,1,\ldots,T$,
where $H_0 = H$ and $A_0=A, B_0 =B$.
Note that the algorithm might fail to produce a perfect matching of $H$,
as there is no guarantee on $H_T$ containing a perfect matching in the final step.
However, in Lemma \ref{lem:final_pm} we will prove that such `bad event'
rarely happens.

\begin{proof}[Proof of Lemma \ref{lem:gnp_r_t_strong} (i)]
We view the probability space
generated by the pair $(G(n,p), H)$ from a slightly different perspective.
Let $p_1=\frac{\varepsilon}{4}p$, and define $p_2$
by $p=p_1+p_2-p_1p_2$. Let $G_{1}=G(n,p_1)$ and $G_{2}=G(n,p_2)$,
and note that $G=G_{1}\cup G_{2}$ has the same distribution as $G(n,p)$.
The random algorithm $\Phi$ can equivalently be defined by
first taking a random subgraph $G_1$,
and then applying a random algorithm $\Psi$, i.e. $\Phi(G) = \Psi(G_1)$.
Further note that all events in the probability space $\mathcal{P}(G, G_1)$ generated
by the pair of graphs $G$ and $G_1$ are measurable
in the probability space $\mathcal{P}(G_1, G_2)$ generated
by the pair of graphs $G_1$ and $G_2$.
Therefore, since the event that we would like to study lies in the probability space
$\mathcal{P}(G, G_1, \Phi)$ we may as well compute its probability
in the space $\mathcal{P}(G_1, G_2, \Psi)$.

Since $G_1$ and $G_2$ are independent, conditioned on $\Psi(G_1) = M$,
the graph $D_{G_{2}}(M)$ has the distribution of the random directed graph
$D(\lfloor \frac{n}{2} \rfloor,p_{2})$.
We thus know by Theorem \ref{thm:di_resilience}
that a.a.s.~every subgraph of $D_{G_{2}}(M)$ of minimum in- and out-degrees at least
$\left(\frac{1}{2}+\frac{\varepsilon}{2}\right)np_2$ is Hamiltonian, i.e.
\[
	\BFP\Big((G_2,\Psi(G_1))\text{ is \ensuremath{\frac{\varepsilon}{2}}-di-ham-resilient} \,\Big|\, \text{$\Psi$ succeeds} \Big) = 1-o(1).
\]
Hence as long as $\Psi$ outputs a perfect matching with high probability
(this fact will be proved in Lemma \ref{lem:final_pm}),
\[
	\BFP\Big((G_2,\Phi(G))\text{ is \ensuremath{\frac{\varepsilon}{2}}-di-ham-resilient} \Big) = 1-o(1).
\]

Observe that if $G_1$ has maximum degree at most $\frac{\varepsilon}{2}np$, then
every subgraph of $D_G(M)$ of minimum in- and out-degrees at least
$\left(\frac{1}{2}+ \varepsilon\right)np$ contains
a subgraph of $D_{G_2}(M)$ of minimum in- and out-degrees at least
$\left(\frac{1}{2}+\frac{\varepsilon}{2}\right)np$. Hence in this case,
$(G_2, \Phi(G))$ being $\frac{\varepsilon}{2}$-di-ham-resilient implies
$(G, \Phi(G))$ being $\varepsilon$-di-ham resilient.
 Let $\mathcal{E}$ be the event that
$(G,\Phi(G))\text{ is $\varepsilon$-di-ham-resilient}$.
Since $G_1$ a.a.s. has maximum degree at most $\frac{\varepsilon}{2}np$,
the observations above imply $\BFP(\mathcal{E})= 1-o(1)$.
Let $\mathcal{R}$ be the collection of graphs $\Gamma$ such that
$\mathbf{P}(\mathcal{E} \,|\, G=\Gamma) = 1 - o(1)$. Then
by $\mathbf{P}(\mathcal{E}) = 1 -o(1)$ and Lemma \ref{lem:two_space_to_one},
we have $\mathbf{P}(G \in \mathcal{R}) = 1 - o(1)$, thus proving the lemma.
\end{proof}

It thus remains to prove Lemma \ref{lem:gnp_r_t_strong} (ii).
Before proceeding further, we establish some simple properties of $G(n,p)$
and $H$ in the following two lemmas. Let $$q = \frac{\varepsilon}{4}p\,.$$
Thus $q$ is the probability that a pair of vertices in $A \times B$
forms an edge in $H$.

\begin{lem} \label{lem:rg_typical_0}
If $p \le \frac{\log^9 n}{n}$, then $G(n,p)$ a.a.s.~has the following property.
For every fixed vertex $v$, there exists at most
one vertex having codegree 2 with $v$, and all other
vertices have codegree at most 1 with $v$.
\end{lem}
\begin{proof}
The claim follows immediately from the easily established (say, through the first moment method) fact that for such values of $p(n)$, the random graph $G(n,p)$ a.a.s.~does not contain two cycles of length 4 sharing a vertex.
\end{proof}

The following lemma involves two layers of randomness; first of $G(n,p)$ and then of the graph $H$.
It asserts that $G(n,p)$  a.a.s.~is chosen so that $H$ (which is determined by another
random event) a.a.s.~has the listed properties.

\begin{lem}\label{lem:rg_typical}
For $p \gg \frac{\log n}{n}$ and fixed reals $\mu,\xi>0$,
$G = G(n,p)$ has the following property with probability $1-o(1)$.
For every $\mu np$-bounded incompatibility system $\mathcal{F}$ over $G$,
the random graph $H$ and the partition $A \cup B$ a.a.s. have the following properties:
\vspace{-0.1cm}
\begin{itemize}
  \setlength{\itemsep}{0pt} \setlength{\parskip}{0pt}
  \setlength{\parsep}{0pt}
\item[(i)] In $H$, all vertices have degree $(1+o(1))n_0q$, and in $G$, all vertices
have degree $(1+o(1))n_0p$ across the partition $A \cup B$.
\item[(ii)] For all $A' \subseteq A$ and $B' \subseteq B$,
$e_H(A',B') = |A'||B'|q + o(n^2q)$.
\item[(iii)] All pairs of sets $A' \subseteq A$ and $B' \subseteq B$
of sizes $|A'|=|B'| \le e^{-2} \xi n_0$ satisfy $e_H(A',B') \le |A'|\cdot \xi n_0 q$.
\item[(iv)] $\mathcal{F}$ induces a $(\mu + \frac{\varepsilon}{5})n_0 q$-bounded incompatibility
system on $H$, and $\mathcal{F}$ induces a  $(\mu + \frac{\varepsilon}{5})n_0 p$-bounded incompatibility system on the bipartite subgraph of $G$ between $A$ and  $B$.
\end{itemize}
\end{lem}
\begin{proof}
Note that the distribution of $H$ is identical to that of the random bipartite graph
with parts of sizes $|A|=|B|=n_0$ obtained by taking each edge independently with probability $q$.
Hence Properties (i) and (ii) follow from Chernoff's inequality, the union bound,
and Lemma \ref{lem:two_space_to_one}. Similarly, Property (iv) follows
from Chernoff's inequality, the concentration of hypergeometric distribution, and the union bound.

To prove Property (iii), note that the probability of a fixed pair of sets $A'$ and $B'$
of size $|A'| = |B'| = k \le e^{-2} \xi n_0$ to satisfy $e_H(A', B') > k \xi n_0 q$ is at most
\[
	{k^2 \choose k \xi n_0 q} q^{k \xi n_0 q}
	\le
	\left( \frac{ek}{\xi n_0} \right)^{k \xi n_0 q}
	\le
	e^{-k \xi n_0 q}
	\ll
	n^{-3k},
\]
where the last inequality follows since $p \gg \frac{\log n}{n}$ and $q = \frac{\varepsilon}{4}p$.
By taking the union bound over all choices of $A'$ and $B'$, we see that the probability
of the existence of a pair of sets $A'$ and $B'$ violating (iii) is at most
$\sum_{k=1}^{e^{-2}\xi n_0} {n \choose k}^2 n^{-3k} \ll 1$.
\end{proof}

Throughout the proof, we will often use the phrase `condition on the
outcome of Lemmas \ref{lem:rg_typical_0} and \ref{lem:rg_typical}', to indicate that we
first condition on $G = G(n,p)$ satisfying Lemmas \ref{lem:rg_typical_0} and
\ref{lem:rg_typical},
and then given a $\mu np$-bounded incompatibility system $\mathcal{F}$
over $G$, condition on $H$ satisfying Lemma \ref{lem:rg_typical}.

\medskip

For $i=0,1,\ldots,T-1$, let $M_i^{(1)} = M_i^{(0)} \setminus M_i$
be the set of edges that were first chosen but then removed at the $i$-th stage.
For a set of vertices $X$, we use the notation $X \cap M_i$
to denote the set of vertices in $X$ that intersect an edge in $M_i$
(similarly define $X \cap M_i^{(0)}$ and $X \cap M_i^{(1)}$).
We also use the notation $x \pm \alpha$ to denote
a quantity between $x - \alpha$ and $x+\alpha$. A combination of two such
estimates $x \pm \alpha = x \pm \alpha'$ means that $|\alpha'| \ge |\alpha|$,
i.e., that the estimate on the right hand side is rougher than that
on the left hand side. The following lemma gives estimate on the number of
edges in $M_i$ intersecting a fixed given set.

\begin{lem} \label{lem:chosen_edge_control}
Condition on the outcome of Lemmas \ref{lem:rg_typical_0} and \ref{lem:rg_typical}.
For an integer $i$ with $0 \le i \le T-1$ and a positive real $\xi_i$
satisfying $\delta \leq \xi_i \le \frac{\varepsilon}{32}$,
suppose that all vertices $x \in V(H_i)$ have degree
$d_{H_i}(x) =  ((1-\delta)^i \pm \xi_i)n_0 q$.
For a vertex $v \in V(H_i)$ and a set $X \subseteq N_{H_i}(v)$
satisfying $|X| \ge 4\delta^{-2}$,
\begin{itemize}
\vspace{-0.1cm}
  \setlength{\itemsep}{0pt} \setlength{\parskip}{0pt}
  \setlength{\parsep}{0pt}
\item[(i)] If $\Gamma$ is a fixed subset of edges incident to $X$, then
$|\Gamma \cap M_i^{(0)}| = (1 \pm 9\varepsilon^{-1} \xi_i)\frac{\delta |\Gamma|}{(1-\delta)^{i}n_0 q}$
with probability $1 - e^{-\Omega(|\Gamma| \delta^3 / n_0q)}$.
\item[(ii)] $|X \cap M_i^{(1)}| \le 5\delta^2|X|$ with probability $1 - e^{-\Omega(\delta^2|X|)}$.
\item[(iii)] $|X \cap M_i| = (1 \pm (15\varepsilon^{-1} \xi_i + 5\delta))\delta |X|$ with probability $1 - e^{-\Omega(\delta^3|X|)}$.
\end{itemize}
\end{lem}
\begin{proof}
The following estimate deduced from $(1-\delta)^{i} \ge (1-\delta)^T = \frac{\varepsilon}{4}$
will be repeatedly used throughout the proof:
\begin{align} \label{eqn:chosen_edge_prob}
	(1-\delta)^{i} \pm \xi_i = (1 \pm 4\varepsilon^{-1}\xi_i)(1-\delta)^{i}.	
\end{align}
Since $m_i = n_i ((1-\delta)^i \pm \xi_i)n_0 q$,  by linearity of expectation, we have
\[
	\BBE[|\Gamma \cap M_i^{(0)}|]
	= |\Gamma| \frac{\delta n_i}{m_i}
	= \frac{\delta}{((1-\delta)^i \pm \xi_i) n_0 q} |\Gamma|
	= (1 \pm 8\varepsilon^{-1} \xi_i) \frac{\delta}{(1-\delta)^i n_0 q} |\Gamma|,
\]
where the last equality follows from \eqref{eqn:chosen_edge_prob}.
Part (i) follows by Chernoff's inequality
since $|\Gamma \cap M_i^{(0)}|$ is a sum of independent random variables (each
indicating whether an edge is chosen or not).

To prove part (ii), recall that $X \subseteq N_{H_i}(v)$ for some vertex $v$.
Let $H_i'$ be the subgraph of $H_i$ obtained by
removing all edges incident to $v$. For each vertex $x \in X$,
let ${\bf 1}_x$ be the indicator random variable of the event that
(a) there exist two edges in $M_i^{(0)} \cap E(H_i')$ incident to $x$, or
(b) there exists a path of length two that has $x$ as its endpoint
and consists of edges in $M_i^{(0)} \cap E(H_i')$.
Let $\Gamma_v$ be the set of edges of $H_i$ incident to $v$, and note that
\begin{align} \label{eq:second_order_estimate}
	|X \cap M_i^{(1)}|
	\le |\Gamma_v \cap M_i^{(0)}| + \sum_{x \in X} {\bf 1}_{x}.
\end{align}
A bound on the first term can be obtained from union bound as follows:
\begin{align*}
	\BFP\left(|\Gamma_v \cap M_i^{(0)}| > \frac{\delta^2}{2}|X|\right)
	&\le
	\sum_{k=\delta^2|X|/2}^{\infty} {d_{H_i}(v) \choose k} \left(\frac{\delta n_i}{m_i} \right)^{k}
	<
	\sum_{k=\delta^2|X|/2}^{\infty} \left( \frac{e\delta n_i d_{H_i}(v)}{m_i k} \right)^{k}.
\end{align*}
Since $m_i = n_i((1-\delta)^i \pm \xi_i)n_0 q$ and $d_{H_i}(v) = ((1-\delta)^i \pm \xi_i)n_0 q$,
it follows from the estimates given above that
\begin{align} \label{eq:soe_secondterm}
	\BFP\left(|\Gamma_v \cap M_i^{(0)}| > \frac{\delta^2}{2}|X|\right)
	\le
	\sum_{k=\delta^2|X|/2}^{\infty} \left( \frac{2e\delta}{ k} \right)^{k}
	< e^{-\Omega(\delta^2|X|)},
\end{align}
where the last inequality follows since $2e\delta / k < 4e / (\delta|X|) \le e\delta \le \frac{1}{2}$.

For the second term on the right-hand-side of \eqref{eq:second_order_estimate},
even though the events $\{{\bf 1}_x\}_{x \in X}$ are not necessarily independent,
we claim that there exists a large subset $X' \subseteq X$
for which they are independent. To see this, note that
since $X \subseteq N_{H_i}(v)$ for some vertex $v$, and
$G(n,p)$ satisfies the event in Lemma~\ref{lem:rg_typical_0},
there exists at most one vertex $w \neq v$ for which $|X \cap N_{H_i}(w)| = 2$ and all
other vertices $w' \neq w,v$ have $|X \cap N_{H_i}(w')| \le 1$. Define $X' = X \setminus N_{H_i}(w)$
if such vertex $w$ exists, and $X' = X$ otherwise (note that $|X| - |X'| \le 2$).
Note that ${\bf 1}_x$ depends only on the set of edges in $H_i'$
that intersect $\{x\} \cup N_{H_i'}(x)$. Since the sets $\{x\} \cup N_{H_i'}(x)$ are disjoint
and $H_i'$ is bipartite,  the events ${\bf 1}_x$ are independent for $x \in X'$
 (conditioned on $G(n,p)$ satisfying the event in Lemma~\ref{lem:rg_typical_0}).
For a fixed $x \in X'$, the probability of event (a) is at most
\[
	{d_{H_i}(x) \choose 2} \Big(\frac{\delta n_i}{m_i}\Big)^2
	\le
	\frac{((1-\delta)^{i} \pm \xi_i)^2 (n_0 q)^2}{2}
		\Big(\frac{\delta}{((1-\delta)^{i} \pm \xi_i) n_0 q}\Big)^2
	\le
	\delta^2,
\]
by \eqref{eqn:chosen_edge_prob} and $\xi_i \le \frac{\varepsilon}{32}$. Similarly, the probability of event (b) is at most

$$
\sum_{y \in N_{H_i}(x)} d_{H_i}(y) \Big(\frac{\delta n_i}{m_i}\Big)^2 \leq
((1-\delta)^{i} \pm \xi_i)^2 (n_0 q)^2 \Big(\frac{\delta}{((1-\delta)^{i} \pm \xi_i) n_0 q}\Big)^2
	\le	2\delta^2.
$$
Therefore $\BBE[{\bf 1}_x] \le 3\delta^2$,
and by Chernoff's inequality we obtain
\begin{align} \label{eq:soe_firstterm}
	\BFP\left(\sum_{x \in X'} {\bf 1}_x \le 4\delta^2|X'|\right) \ge 1 - e^{-\Omega(\delta^2|X'|)}.
\end{align}
Since $|X| \le |X'| + 2$ and $|X| \ge 4\delta^{-2}$, part (ii) follows from
\eqref{eq:second_order_estimate}, \eqref{eq:soe_secondterm}, and \eqref{eq:soe_firstterm}.

To prove part (iii), let $\Gamma_X$ be the set of edges incident to $X$, and
note that $|\Gamma_X| = ((1-\delta)^i \pm \xi_i)n_0 q |X|$. By part (i), we have
$|\Gamma_X \cap M_i^{(0)}| = (1 \pm 15\varepsilon^{-1} \xi_i) \delta|X|$
with probability $1 - e^{-\Omega(\delta^3|X|)}$. By part (ii), we have
$|X \cap M_i^{(1)}| \le 5\delta^2|X|$
with probability $1 - e^{-\Omega(\delta^2|X|)}$. Now part (iii) follows since
\[
	|\Gamma_X \cap M_i^{(0)}| - |X \cap M_i^{(1)}|
	\le
	|X \cap M_i|
	\le
	|\Gamma_X \cap M_i^{(0)}|. \qedhere
\]
\end{proof}

In order to prove Lemma \ref{lem:gnp_r_t_strong} (ii), we need to understand how the
edges of $D_G(M)$ get removed in $D_G(M;\mathcal{F})$. In particular,
we need to track these changes with each iteration of the random algorithm.
The following technical definitions are made with this purpose in mind.

Suppose that an instance $G$ of $G(n,p)$ and a $\mu np$-bounded incompatibility system
$\mathcal{F}$ over $G$ are fixed.
Let $e = \{a, b\} \in E(G)$ and $e' = \{a', b'\} \in E(G)$ be two edges such that
$a,a' \in A$ and $b, b' \in B$. We say that $e'$ is
{\em $A$-bad for $e$} if $\{a, b'\} \in E(G)$ is an edge compatible
with $e$, but incompatible with $e'$. Similarly, we say that $e'$
is {\em $B$-bad for $e$} if $\{b, a'\} \in E(G)$ is an edge compatible
with $e$, but incompatible with $e'$.

For an edge $e = \{a,b\}$ with $a \in A$ and $b \in B$, define
\begin{align*}
	A_e^{(G)} &= \{ x \in A : \{b,x\} \in E(G), \{b,x\} \text{ and } \{a,b\} \text{ are compatible} \}
	\quad \text{and} \\
	B_e^{(G)} &= \{ y \in B : \{a,y\} \in E(G), \{a,y\} \text{ and } \{a,b\} \text{ are compatible} \}.
\end{align*}
Similarly define $A_e^{(H)}$ and $B_e^{(H)}$ by considering the edges of $H$
instead of the edges of $G$ in the definition above.

\begin{rem}
For odd $n$, we need to extend the definitions
above to edges whose both endpoints are in $A$. In this case, for an edge $e = \{a_1, a_2\}$ with $a_1, a_2 \in A$, we consider the edge as an ordered pair $(a_1, a_2)$ to distinguish
$(a_1, a_2)$ and $(a_2, a_1)$, and define an edge $e'=\{a',b'\}$ (where $a' \in A$ and $b' \in B$) 
to be {\em $A$-bad for $(a_1, a_2)$} if $\{a_1, b'\} \in E(G)$ is compatible with $e$ but
incompatible with $e'$ (there is no need to consider $B$-bad edges for $e$).
Also, instead of having a pair of sets $A_{e}^{(G)}$ and $B_{e}^{(G)}$ as above,
we define two sets $B_{(a_1, a_2)}^{(G)}$ and $B_{(a_2, a_1)}^{(G)}$ 
as 
\begin{align*}
	B_{(a_1, a_2)}^{(G)} &= \{ y \in B : \{a_1,y\} \in E(G), \{a_1,y\} \text{ and } \{a_1,a_2\} \text{ are compatible} \}
	\quad \text{and} \\
	B_{(a_2, a_1)}^{(G)} &= \{ y \in B : \{a_2,y\} \in E(G), \{a_2,y\} \text{ and } \{a_2,a_1\} \text{ are compatible} \}.
\end{align*}
Similarly define the sets $B_{(a_1, a_2)}^{(H)}, B_{(a_2, a_1)}^{(H)}$.
As this modified definition becomes relevant only in few places, we
will abuse notation and use $A$-bad edges and $B_e^{(G)}, B_e^{(H)}$ to denote
both ordered pairs $(a_1, a_2)$ and $(a_2, a_2)$.
\end{rem}

Define
\[
	\xi_i = \delta (1+21\varepsilon^{-1}\delta)^{i}.
\]
For each $i = 0,1,\ldots, T$, we say that
$H_i = H[A_i \cup B_i]$ is {\em normal} if
\begin{itemize}
  \setlength{\itemsep}{0pt} \setlength{\parskip}{0pt}
  \setlength{\parsep}{0pt}
\item[(i)] $n_i = |A_i| = |B_i| = \Big((1-\delta)^{i} \pm \xi_i \Big) n_0$.
\item[(ii)] For all vertices $v \in V(H_i)$, $d_{H_i}(v) = \Big((1-\delta)^{i} \pm \xi_i\Big) n_0 q$.
\item[(iii)] For all vertices $v \in V(G)$, $|N_{G}(v) \cap A_i| = \Big((1-\delta)^{i} \pm \xi_i\Big) n_0 p$ and $|N_{G}(v) \cap B_i| = \Big((1-\delta)^{i} \pm \xi_i\Big) n_0 p$.
\item[(iv)] For all edges $e \in E(G)$, $|A_e^{(G)} \cap A_i| = \Big((1- \delta)^{i} \pm\xi_i \Big) |A_e^{(G)}|$, and a similar estimate holds for the sets $B_e^{(G)}, A_e^{(H)}, B_e^{(H)}$.
\item[(v)] If $i\neq 0$, then for all edges $e \in E(G)$, there are at most
$\Big(\mu + \frac{\varepsilon}{3}\Big) \delta(1-\delta)^{i-1} |B_e^{(G)}|$
$A$-bad edges for $e$ in $M_{i-1}$ (a similar estimate for $B$-bad edges).
\end{itemize}
Note that since $T = \frac{\ln(4/\varepsilon)}{-\ln(1-\delta)} \le \frac{\ln(4/\varepsilon)}{\delta}$
and $\delta = e^{-22\varepsilon^{-1} \ln \varepsilon^{-1}}$, the error parameter $\xi_i$ satisfies
\begin{align} \label{eq:error_param_bound}
	\xi_i \le \xi_T
	\le \delta(1 + 21\varepsilon^{-1}\delta)^{T}
	\le \delta e^{21\varepsilon^{-1} \ln(4/\varepsilon)}
	\le e^{-\varepsilon^{-1}}.
\end{align}
The following lemma asserts that with high probability all $H_i$ are normal
(for $i=0,1,\ldots, T$).

\begin{lem} \label{lem:induction_step}
Conditioned on the outcome of Lemmas \ref{lem:rg_typical_0} and \ref{lem:rg_typical},
the graph $H_i$ is a.a.s.~normal for all $i=0,1,\ldots,T$.
\end{lem}
\begin{proof}
We proceed by induction on $i$. For $i=0$, the statement follows immediately
since we conditioned on Lemma \ref{lem:rg_typical}.
For $i \ge 0$, suppose that $H_i$ is normal.
For each vertex $v \in V(H_i)$, since $d_{H_i}(v) = ((1-\delta)^i \pm \xi_i) n_0 q$,
by applying Lemma \ref{lem:chosen_edge_control} (iii)
with $X = N_{H_i}(v)$, we see that
\begin{eqnarray*}
	|N_{H_i}(v) \cap M_i|
	&=& (1 \pm (15\varepsilon^{-1}\xi_i + 5\delta)) \delta ((1-\delta)^{i} \pm \xi_i)n_0 q\\
	&=&(1 \pm (15\varepsilon^{-1}\xi_i + 5\delta)) \delta (1 \pm 4\varepsilon^{-1}\xi_i)(1-\delta)^{i} n_0 q\\
	&=& (1 \pm 20\varepsilon^{-1}\xi_i) \delta (1-\delta)^{i} n_0 q
\end{eqnarray*}
with probability $1 - e^{-\Omega(\delta^3(1-\delta)^{i}n_0q)} = 1 - n^{-\omega(1)}$. Since
\begin{align*}
	d_{H_{i+1}}(v)
	&= d_{H_i}(v) - |N_{H_i}(v) \cap M_i| = ((1-\delta)^{i} \pm \xi_i) n_0 q - (1 \pm 20\varepsilon^{-1}\xi_i) \delta (1-\delta)^{i} n_0 q \\
	&= (1-\delta)^{i+1} n_0 q \pm (1 + 20 \varepsilon^{-1}\delta)\xi_i n_0 q=((1-\delta)^{i+1} \pm \xi_{i+1})n_0q,
\end{align*}
by taking the union bound over all vertices,
we see that Property (ii) of $H_{i+1}$ being normal a.a.s. holds.
Properties (iii) and (iv) follow by the same argument applied to the
corresponding sets. Furthermore, if Property (ii) holds, then $H_{i+1}$ is a balanced
bipartite graph with $n_{i+1}$ vertices in each part
whose number of edges is
\[
	m_{i+1} = n_{i+1} \Big((1-\delta)^{i+1}  \pm (1 + 20 \varepsilon^{-1}\delta)\xi_i\Big)n_0 q.
\]
By Lemma \ref{lem:rg_typical} (iii), we have that $n_{i+1}$ is linear in $n$. Then part (ii) of the same lemma implies that
$n_{i+1} = ((1-\delta)^{i+1} \pm \xi_{i+1})n_0$, proving Property (i).

For a fixed edge $e=\{a,b\} \in E(G)$ with $a \in A$ and $b \in B$,
let $\Gamma_{e,A}$ be the set of $A$-bad edges for $e$ in $H_i$.
Each $A$-bad edge in $H_i$ can be accounted for by first taking
a vertex $x \in B_e^{(G)} \cap B_i$, and then counting the number of edges $\{x,y\} \in E(H_i)$
that are incompatible with $\{a,x\}$. This gives
\[
	|\Gamma_{e,A}|
	= \sum_{x \in B_{e}^{(G)} \cap B_i} \Big(d_{H_i}(x) - |A_{\{x,a\}}^{(H)} \cap A_i|\Big)
	= ((1-\delta)^i \pm \xi_i) \sum_{x \in B_{e}^{(G)} \cap B_i} \Big(n_0 q - |A_{\{x,a\}}^{(H)}|\Big).
\]
By Lemma \ref{lem:rg_typical} (i) and (iv), we have
$n_0 q - |A_{\{x,a\}}^{(H)}| \le (\mu + \frac{\varepsilon}{4})n_0 q$, which in turn gives
\begin{align*}
	|\Gamma_{e,A}|
	&\le ((1-\delta)^i \pm \xi_i) \sum_{x \in B_{e}^{(G)} \cap B_i} \Big(\mu + \frac{\varepsilon}{4}\Big) n_0 q
	= ((1-\delta)^i \pm \xi_i)^2 \Big(\mu + \frac{\varepsilon}{4}\Big) n_0 q |B_e^{(G)}|.
\end{align*}
By Lemma \ref{lem:chosen_edge_control} (i), with probability $1 - n^{-\omega(1)}$,
\begin{align*}
	|\Gamma_{e,A} \cap M_i|
	\le |\Gamma_{e,A} \cap M_i^{(0)}|
	\le (1 + 9\varepsilon^{-1}\xi_i) \frac{ \delta |\Gamma_{e,A}|}{(1-\delta)^{i}n_0q}
	\le
	(1 + 20\varepsilon^{-1}\xi_i) \Big(\mu + \frac{\varepsilon}{4}\Big) \delta (1-\delta)^{i} |B_e^{(G)}|.
\end{align*}
Equation \eqref{eq:error_param_bound} implies $(1 + 20\varepsilon^{-1}\xi_i)(\mu + \frac{\varepsilon}{4}) \le \mu + \frac{\varepsilon}{3}$, and thus Property (v) for $A$-bad edges follows
by taking the union bound over all edges. The conclusion for $B$-bad edges follows
by a similar argument.
\end{proof}

We show that $\Phi$ successfully terminates as long as the final iteration is normal.

\begin{lem} \label{lem:final_pm}
Conditioned on the outcome of Lemmas \ref{lem:rg_typical_0} and \ref{lem:rg_typical},
if $H_T$ is normal, then it contains a perfect matching.
\end{lem}
\begin{proof}
If $n$ is odd, then we must first choose the vertices
$v_*, a_{(n-1)/2}$ from $A_T$ and $b_{(n-1)/2}$ from $B_T$. Property (iii) of
normality implies the existence of an edge $\{v_*, a_{(n-1)/2}\} \in E(G)$ within $A_T$.
Then by Properties (ii) and (iv) of normality, we can find a vertex
$ b_{(n-1)/2} \in B_T$ for which $\{v_*, a_{(n-1)/2}\}$ and
$\{v_*, b_{(n-1)/2}\} \in E(H_T)$ are compatible.
Remove the vertices $v_*, a_{(n-1)/2}$ from $A_T$ and $b_{(n-1)/2}$ from $B_T$
and update $n_T$ as the size of the new sets $A_T$ and $B_T$.

Now consider both even and odd $n$, and
consider the graph $H_{T}$ with bipartition $A_{T} \cup B_{T}$.
By Hall's theorem it suffices to prove that
\[
	\forall A' \subseteq A_T, |A'| \le \frac{1}{2}n_T \quad |N_{H_T}(A')| \ge |A'|
	\quad \text{and} \quad
	\forall B' \subseteq B_T, |B'| \le \frac{1}{2}n_T \quad |N_{H_T}(B')| \ge |B'|.
\]
Suppose that there exists a set $A' \subseteq A_T$ for which
$|N_{H_T}(A')| < |A'|$, and let $X$ be a superset of $N_{H_T}(A')$ of size exactly
$|A'|$. Since $H_T$ is normal, by Property (ii) of normality, we see that
\[
	e_{H_T}(A', X) = e_{H_T}(A', B_T) \ge |A'|((1-\delta)^T - \xi_i) n_0q \ge |A'| \frac{\varepsilon}{8} n_0 q.
\]
By Lemma \ref{lem:rg_typical} (iii), this implies $|A'| > \frac{\varepsilon}{8e^2} n_0$.
Let $Y=B_T-X$. Then $|Y|\geq \frac{1}{2}n_T$ and by definition $e_{H_T}(A', Y)=0$. On the other hand,
by Lemma \ref{lem:rg_typical} (ii), $e_{H_T}(A', Y) = |A'||Y| q + o(n^2q)>0$.
This contradiction proves that no such set $A'$ can exist. Similarly, we can show that no such set $B'$ exists.
 \end{proof}

We conclude this section with the proof of the second part of Lemma \ref{lem:gnp_r_t_strong}.

\begin{proof}[Proof of Lemma \ref{lem:gnp_r_t_strong} (ii)]
Let $\mu = 1 - \frac{1}{\sqrt{2}} - 2\varepsilon$ and
condition on $G = G(n,p)$ satisfying Lemmas \ref{lem:rg_typical_0} and \ref{lem:rg_typical}.
Since $G(n,p)$ satisfies these lemmas
with probability $1-o(1)$, it suffices to prove that for every
$\mu np$-bounded incompatibility system $\mathcal{F}$ over $G$,
the probability that $(G,\Phi(G), \mathcal{F})$ is $\varepsilon$-typical is $1-o(1)$.
This will be achieved by proving that $(G, \Phi(G), \mathcal{F})$ is $\varepsilon$-typical
if $H$ satisfies Lemma \ref{lem:rg_typical} and
all $H_i$ are normal (note that all events have probability $1-o(1)$).

Since $H_T$ is normal, Lemma \ref{lem:final_pm} implies
that our algorithm produces a perfect matching $M = \Phi(G)$.
By Lemma \ref{lem:rg_typical} (i), each vertex in $D_G(M)$ has in- and out-degrees $(1+o(1))n_0 p$.
Take a vertex $e = \{a,b\} \in D_G(M)$ .
An out-neighbor $f = \{a',b'\}$ of $e$ in $D_G(M)$
can be removed in $D_G(M;\mathcal{F})$ for two reasons: first, if
$\{b,a'\}$ is an edge incompatible with $e$, and second, if
$\{b,a'\}$ is an edge compatible with $e$ but incompatible with $f$ (i.e. $f$ is a $B$-bad edge for $e$).
By the definition of the set $B_e^{(G)}$, we see that
there are at most $(1+o(1))n_0 p - |B_e^{(G)}|$ edges of the first type.
Since all $H_i$ are normal, by conditions (iii) and (v) of normality,
the number of edges of the second type is at most
\begin{align*}
	 |N_{G}(a) \cap B_T| + \sum_{i=0}^{T-1} \Big(\mu + \frac{\varepsilon}{3}\Big) \delta (1-\delta)^{i} |B_e^{(G)}|
	 &\le
	 ((1-\delta)^T + \xi_T)n_0 p + \Big(\mu + \frac{\varepsilon}{3}\Big) |B_e^{(G)}|.
\end{align*}
Therefore, since $(1-\delta)^{T} = \frac{\varepsilon}{4}$,
the total number of out-edges removed from $e$ is at most
\[
	\Big((1+o(1))n_0 p - |B_e^{(G)}|\Big) + \frac{\varepsilon n_0 p}{3} + \Big(\mu + \frac{\varepsilon}{3}\Big) |B_e^{(G)}|,
\]
which by $|B_e^{(G)}| \ge (1-\mu -o(1))n_0 p$ and $\mu = 1 - \frac{1}{\sqrt{2}} - 2\varepsilon$ is at most
\begin{align*}
	& (1+o(1))n_0 p + \frac{\varepsilon n_0 p}{3} + \Big(-1 + \mu + \frac{\varepsilon}{3}\Big)(1 - \mu - o(1))n_0 p
	\\
	\le\,&
	\left(1 + \frac{\varepsilon}{3} - (1 - \mu)^2 + \frac{\varepsilon}{3}(1 - \mu) + o(1) \right)n_0 p
	\le
	\left(\frac{1}{2} - 2\varepsilon \right)n_0 p.
\end{align*}
Hence the minimum out-degree of $D_G(M;\mathcal{F})$ is at least
$(1+o(1))n_0 p - (\frac{1}{2} - 2\varepsilon)n_0 p \ge (\frac{1}{2} + \varepsilon) n_0 p$.
A similar bound on the minimum in-degree of $D_G(M;\mathcal{F})$ holds.
\end{proof}

\section{Concluding remarks}

In this paper, we proved the existence of a positive real $\mu$
such that if $p \gg \frac{\log n}{n}$, then
$G=G(n,p)$ a.a.s.~has the following property. For every $\mu np$-bounded
incompatibility system $\mathcal{F}$ defined over $G$, there
exists a Hamilton cycle in $G$ compatible with $\mathcal{F}$.
The value of $\mu$ that we obtained in Theorem~\ref{thm:sparse_p}
is very small, but for $p \gg \frac{\log^8 n}{n}$ we improved it in Theorem \ref{thm:dense_p}
to $1 - \frac{1}{\sqrt{2}} - o(1)$ (roughly 0.29). The bound of
$\frac{\log^8 n}{n}$ came from Theorem \ref{thm:di_resilience}, and in fact,
any improvement in the range of probability for Theorem \ref{thm:di_resilience}
will immediately imply Theorem \ref{thm:dense_p} for the extended range of probabilities.
It is not clear what the best possible value of $\mu$ should be. 
The example of Bollob\'as and Erd\H{o}s \cite{BoEr76} of a
$\lfloor \frac{1}{2}n \rfloor$-bounded edge-coloring of $K_n$ with no
properly colored Hamilton cycles implies that the optimal value of
$\mu$ is at most $\frac{1}{2}$, since it provides an upper bound for the case $p=1$.

\medskip

The concept of incompatibility systems seems to give an interesting
new take on robustness of graphs properties. Further study of how various extremal results
can be strengthened using this notion appears to be a promising direction of research.
For example in a companion paper \cite{KLS-Dirac}, we show that there exists a constant $\mu>0$ such that for
any $\mu n$-bounded system $\mathcal{F}$ over a graph $G$ on $n$ vertices with minimum degree at least $n/2$, there is a compatible
Hamilton cycle in $G$. This establishes in a very strong sense an old conjecture of
H\"aggkvist  from 1988.

\vspace{0.4cm}
\noindent
{\bf Acknowledgments.} 
We thank Asaf Ferber for calling our attention to the problem of rainbow Hamilton cycles in random graphs.
A major part of this work was carried out when Benny Sudakov was visiting Tel Aviv University, Israel.
He would like to thank the School of Mathematical Sciences of Tel Aviv University for hospitality and for creating a stimulating research
environment.

\end{document}